\documentclass[11pt,a4paper]{article}
\usepackage[utf8]{inputenc}
\usepackage[english]{babel}
\usepackage{mathrsfs}
\usepackage{titlesec}
\usepackage{array}
\usepackage{amsmath}
\usepackage{amsthm}
\usepackage{mathtools}
\usepackage{amsfonts}
\usepackage{amssymb}
\usepackage{graphicx}
\usepackage{float}
\usepackage[font=small,labelfont=bf]{caption}
\usepackage{tikz}
\usetikzlibrary{matrix,positioning,arrows.meta, shapes,positioning,chains, calc, decorations.pathmorphing}
\tikzset{toarrow/.style={{|[scale=0.7]}-{>[scale=0.7]}},
    backarrow/.style={{<[scale=0.7]}-{|[scale=0.7]}}}
\usepackage{dsfont}
\usetikzlibrary{matrix}
\usepackage{tikz-cd}
\usepackage{xfrac}
\usepackage{indentfirst}
\usepackage{amsmath,amssymb, amsthm, amscd, mathrsfs}
\usepackage[nottoc]{tocbibind}
\usepackage{hyperref, bookmark}
\usepackage{cleveref}
\usepackage{xcolor}
\usepackage{MnSymbol}
\usepackage{thmtools}
\usepackage{thm-restate}
\usepackage[titletoc,toc,title]{appendix}

\usepackage{tikz}
\usetikzlibrary{matrix,arrows,decorations.pathmorphing}

{\left\lbrace\begin{array}{@{}l@{}}}%
{\end{array}\right.}

\newcommand{\R}{\mathbb{R}}
\newcommand{\Q}{\mathbb{Q}}
\newcommand{\N}{\mathbb{N}}
\newcommand{\Z}{\mathbb{Z}}
\newcommand{\C}{\mathbb{C}}

\usepackage[left=2cm,right=2cm,top=2.5cm,bottom=2.5cm]{geometry}

\theoremstyle{plain}
\newtheorem{theorem}{Theorem}[section]

\newtheorem{lemma}[theorem]{Lemma}

\newtheorem{prop}[theorem]{Proposition}
\newtheorem{cor}[theorem]{Corollary}
\newtheorem*{theorem*}{Theorem}

\theoremstyle{definition}
\newtheorem{defi}[theorem]{Definition}

\theoremstyle{remark}
\newtheorem{ex}[theorem]{Example}
\newtheorem{rem}[theorem]{Remark}

\usepackage{relsize}
\usepackage{slashed}

\newcommand{\bigslant}[2]{{\raisebox{.4em}{$#1$}\left/\raisebox{-.4em}{$#2$}\right.}}

\newcommand\restr[2]{{
  \left.\kern-\nulldelimiterspace 
  #1 
  \littletaller 
  \right|_{#2} 
  }}

\newcommand{\littletaller}{\mathchoice{\vphantom{\big|}}{}{}{}}

\author{
  Pedrotti, Riccardo\\
  \texttt{pedrotti.riccardo@math.utexas.edu}
 }
\title{Fixed point Floer cohomology of disjoint Dehn twists on a $w^+$-monotone manifold with rational symplectic form}
\makeindex

\begin{document}
\maketitle
\begin{abstract}
In this paper we give an explicit description of the Floer cohomology of a composition of Dehn twists $\tau$ about disjoint Lagrangian spheres in a $w^+$-monotone symplectic manifold whose symplectic class  $[\omega]$ admits a rational representative.  To do so,  we generalize the approach developed in \cite{Sei96} and \cite{Gau03} and apply it to the \textit{modified} Floer cohomology groups defined by K. Ono in \cite{Ono95} which can be shown to be isomorphic to the standard ones.  As a byproduct of this new framework,  in a monotone manifold, we are able to prove that a certain class in the fixed point Floer cohomology group of a single Dehn twist $\tau_V$ has to vanish. This class, which counts pseudo-holomorphic half-strips bound to the Lagrangian sphere $V$, plays a role in a new geometric proof of the exactness of the triangle P. Seidel defined in \cite{Sei03} which is the subject of subsequent ongoing work.
\end{abstract}

\section{Introduction}
Let $(M,\omega)$ be a $w^+$-monotone rational symplectic closed manifold. In this paper we are interested in computing the fixed point Floer cohomology for a composition of Dehn twists $\tau$ about disjoint Lagrangian spheres in $M$.  In \cite{Sei97},  P.  Seidel studied them extensively and used them to exhibit an example of a symplectomorphism that is smoothly isotopic to the identity but not symplectically isotopic to it. In there, the fixed point Floer homology of the Dehn twist in a (weakly-monotone) $4$-dimensional symplectic manifold is computed as an abstract quotient of the quantum cohomology ring $QM^*(M,\omega)$ (\cite[Theorem 3.5]{Sei97}), leveraging a certain exact sequence (\cite[Theorem 16.1]{Sei97}) which was then  greatly generalised (\cite{Sei03},\cite{Sei08}) in what is now known as the Seidel's long exact sequence/triangle. In particular, in his thesis, P. Seidel put a lot of emphasis to the quantum module structure of the groups involved which turned out to be crucial for his result.\\ Under the assumption that the \textit{PSS} isomorphism for $HF^*(V,V)$ can be defined, our main result then follows from the existence of such triangle. Usually this involves requiring that $\pi_2(M,V)=0$, and it is no longer true under the more general hypothesis of monotonicity of $V$. For an overview on this issue, see \cite{Alb08}.\\
Since we are not interested in the quantum module structure of $HF^*(\tau_V)$, our approach sidesteps the need of invoking the Lagrangian \textit{PSS} map and provides an explicit isomorphism between $HF^*(\tau_V)$ and $H^*(M,V)$ without relying on Seidel's exact triangle for fixed point Floer cohomology (see \cite[Theorem 4.2 and Example 4.3]{Sei01}). Other than not specialising specifically to $4$-dimensional manifolds, this is the main difference between our approach and the one in \cite{Sei97}, and it is based on the observation that if one can rule out certain pseudo-holomorphic strips, then the fixed point Floer cochain complex of $\tau$ naturally coincides with a cochain complex computing the relative Morse cohomology of a pair of subspaces of $M$.\\
The assumption on $M$ being $w^+$-monotone is there to avoid using virtual techniques, and it is automatically satisfied for manifolds of dimension at most $4$ for example,  while the rationality of the symplectic form is used to prove invariance under continuation maps of the \textit{modified} Floer cohomology introduced in \cite{Ono95}. While in the Hamiltonian setting we are allowed to perturb the symplectic form to achieve its rationality, in fixed point Floer cohomology we must make sure that the given composition of twists $\tau$ remains a symplectomorphism after the perturbation. If the dimension of $M$ is at least $6$, we show that this can be done and therefore the rationality assumption on $[\omega]$ can be safely dropped.\\
In subsequent work we plan on using this framework to give a more geometric construction of the exact triangle and use it to obtain a combinatorial formula to count certain pseudo-holomorphic sections.  As a necessary technical step to do so, one has to study the moduli space of pseudo-holomorphic half-strips bound to a framed Lagrangian sphere $V$. We define a cocycle $c \in CF^*\left(\tau_V; \Lambda_{\omega}\right)$ in terms of a count of elements of such moduli space and exploit our framework to prove it is null-cohomologous.\\

The main results of this paper are the following:
\begin{theorem}\label{theorem 1 case 2}
Let $(M,\omega)$ be a closed $w^+$-monotone symplectic manifold of dimension $2n\geq 4$ such that the symplectic class $[\omega]$ admits a rational representative.  Let $V_1,\dots, V_l,\dots, V_m$ be pairwise disjoint framed Lagrangian spheres and set $\tau:=\tau_{V_1}^{\sigma_1}\cdots\tau_{V_l}^{\sigma_l}\cdots \tau_{V_m}^{\sigma_m}$,  where $\sigma_l=\pm 1$ for all $l$.  Let $C_+$ (resp. $C_-$) be the union of all $V_i$'s such that $\sigma_l=+1$ (resp. $\sigma_l=-1$), 
then
\begin{equation*}
    HF^{k}(\tau; \Lambda_{\omega}) \cong  \bigoplus_{j=k \pmod{2}} H^{j}\left(M \setminus C_-, C_+; \Lambda_{\omega}\right)
\end{equation*}
where the Floer cohomology of $\tau$ is $\Z_2$-graded and $\Lambda_{\omega}$ is the Novikov field associated to the symplectomorphism $\tau : M \to M$.  The isomorphism is (up to a shift) of $\Z_2$-graded $\Lambda_{\omega}$-vector spaces.
\end{theorem}
\begin{rem} By following \cite[Page 589]{DoSa94}, if $N$ is the Chern number of $M$, we can define a \textit{relative} $\Z_{2N}$-grading on $HF^*(\tau; \Lambda_{\omega})$ (i.e. defined up to a global shift). By choosing suitable conventions (\cite[Section 6]{HS95}) our isomorphism is easily seen to preserve this relative grading.\\ If $\dim M \geq 6$, we can drop the assumption on the rationality of $[\omega]$ since we show we can perturb $\omega$ to a make it rational while preserving the $V_l$'s as Lagrangian spheres.
\end{rem}
\begin{theorem}\label{theorem vanishing c}
Let $(M,\omega)$ be a closed monotone symplectic manifold of dimension $2n\geq 4$.  Let $V$ be a framed Lagrangian sphere and let $\tau_V^{\sigma}$, for $\sigma = \pm 1$,  be the Dehn twist (or its inverse) around it.  We can define a cocycle $c \in CF^*(\tau_V^{\sigma}; \Lambda_{\omega})$ by counting index $0$ pseudo-holomorphic half-strips bound to $V$.  It satisfies \[ [c]=0 \in HF^*(\tau_V^{\sigma}; \Lambda_{\omega}).\]
\end{theorem}
Notice that we are not requiring $[\omega]$ to be rational for the second theorem but we impose a monotonicity condition on $M$.\\
We have an analogous statement in the case of a Riemann surface $(\Sigma,\omega)$:
\begin{theorem}\label{theorem vanishing c for surfaces}
Let $(\Sigma,\omega)$ be a closed Riemann surface of genus at least $2$.  Let $V$ be a framed non-contractible Lagrangian sphere and let $\tau_V^{\sigma}$, for $\sigma = \pm 1$,  be the Dehn twist (or its inverse) around it.  We can define a cocycle $c \in CF^*(\tau_V^{\sigma}; \Z_2)$ by counting index $0$ pseudo-holomorphic half-strips bound to $V$.  Then 
\[ [c]=0 \in HF^*(\tau_V^{\sigma}; \Z_2).\]
\end{theorem}

As a corollary of Theorem \ref{theorem 1 case 2},  we have a quick proof of a well-known fact proved in \cite{FOOO10} or \cite{Alb05} in our particular setting:
\begin{cor}
Let $V$ be a Lagrangian sphere in a $w^+$-monotone closed manifold $M$ of dimension $2n\geq 6$.  If $[V]\neq 0$ in $H_n(M;\Z_2)$,  then $V$ cannot be displaced via an Hamiltonian isotopy.
\end{cor}
The corollary is a direct consequence of the fact that if $V$ could be displaced to a disjoint Lagrangian sphere $V'$ via a Lagrangian isotopy,  then 
\[HF^{\bullet}(\tau_V\circ \tau_{V'}^{-1}; \Lambda_{\omega})\cong HF^{\bullet}(\text{Id}_M; \Lambda_{\omega}).\] On the other hand, a standard computation with $\Z_2$-coefficients cohomology shows that,  as $\Z_2$-graded $\Z_2$-vector spaces \[{H^{\bullet}(M\setminus V',V; \Z_2)\not\cong H^{\bullet}(M; \Z_2)},\] which implies
\[H^{\bullet}(M\setminus V',V; \Z_2)\otimes_{\Z_2} \Lambda_{\omega}\not\cong H^{\bullet}(M; \Z_2)\otimes_{\Z_2} \Lambda_{\omega}\]
contradicting Theorem \ref{theorem 1 case 2}.\\

The proof of the Theorem \ref{theorem 1 case 2} roughly follows the reasoning done in \cite{Sei96} with some key differences due to the different setting we are working on.  Without a (strong) monotonicity of $M$,  there are no obvious \textit{a priori} upper-bounds on the energy of low-index solutions of the Floer-Cauchy-Riemann equations.  For this reason we construct an energy filtration on $CF^{\bullet}(\tau)$ and prove the theorem for each subcomplex of the filtration.  An analogue filtration was intensively studied in \cite{Ono95} in order to define the so called \textit{modified} Floer cohomology and in \cite{FOOO10} in order to construct a spectral sequence for Hamiltonian Floer cohomology.\\
By construction,  at each step of the filtration we have an \textit{a priori} bound on the energy of Floer trajectories realizing nonzero differentials.  Thanks to those bounds we can exploit a similar idea from \cite{Gau03},  where the author proves that by arbitrarily rescaling the symplectic form in a tubular neighborhood of certain circles,  pseudo-holomorphic strips passing through them must have arbitrarily high lower-bounds on their energy.  Since we are working in dimension greater than $2$,  we cannot perturb the symplectic form on certain compact neighborhoods of our Lagrangian sphere since that would violate the closedness of our symplectic form. Instead we used a neck-stretching technique borrowed from symplectic field theory.  Thanks to that we created a family of compatible almost complex structures on $(M, \omega)$ (together with a family of induced metrics) which force the energy of pseudo-holomorphic strips passing through these neighborhoods to be arbitrarily high. 
In this way,  trajectories that go through the neck of the Lagrangian spheres $V_i$'s will not contribute to differentials on certain levels of the filtration.  A diagonal argument, together with some technical results,  is then needed to show that such reasoning can be pushed to the limit in order to retrieve Theorem \ref{theorem 1 case 2} for the \textit{modified} flavour of fixed point Floer cohomology.  In the appendix, we included an adaption of unpublished work of K. Ono which shows that the \textit{modified} Floer cohomology theory computed using this energy filtration is canonically isomorphic to the standard one. \\
The last part of this paper is dedicated to adapt our neck-stretching and filtration techniques to the case of half-strips that concerns Theorem \ref{theorem vanishing c} and its version for surfaces.  Most notably,  we noted that no additional requirement on the symplectic form is required other than $M$ being monotone on spheres is necessary.  
\section*{Acknowledgments}
First and foremost I would like to thank my advisor,  Timothy Perutz,  for his continued support and guidance.  After explaining to him the main ideas of an early version of this work, he pushed me to try to generalize it to a broader class of manifolds, leading to the current statement of the main results.  I am also very grateful to Kaoru Ono,  for his careful explanation of his work involving modified Floer cohomology,  the isomorphism with the standard Floer cohomology and other technicalities.  I want to thank my former colleagues Arun Debray and Ivan Tulli for their helpful comments on preliminary ideas on this project.  Lastly,  I want to express my gratitude to the anonymous referee for their meticulous work on this manuscript and for pointing out some gaps in the original version of this article.
\newpage
\tableofcontents

\section{Setting up the Floer cohomology groups}\label{setting up the floer cohomology groups}
The definition of the Floer cohomology groups for a symplectomorphism 	\[\phi : (M,\omega) \to (M,\omega)\] is by now well understood and can be done in rather great generality.  Nowadays there are two main approaches one can take in order to set up the machinery: the Morse-Novikov homology of the twisted loop space or the symplectic mapping torus and horizontal sections viewpoint. The former was historically the first one developed since it is closely related to the classical Hamiltonian Floer (co)homology construction (see for example \cite{DoSa94} and \cite{Lee05}).  The (co)chain complex is generated by the critical points of the action functional defined on some covering of the twisted loop space while the differentials are represented by the gradient flow lines of such function. \\ The definition of Floer cohomology in terms of horizontal sections of the symplectic mapping torus and pseudo-holomorphic sections of a Lefschetz fibration is due to P.  Seidel and has the benefit of being very geometric (see \cite{Sei97} for the general construction).  The (co)chain complex is generated by horizontal sections of the mapping torus $M_{\phi}$ and the differential is given by counts of pseudoholomorphic sections of the product bundle 
\[ E = \R \times M_{\phi} \to \R \times S^1.\] 
For the purpose of this paper we decided to follow the approach in \cite{Lee05} due to our necessity of working with the action functional to define our energy filtration on the cochain complex.  We summarize the main steps of the construction here:\\
To avoid the use of virtual techniques to establish compactness of the relevant moduli spaces,  invariance under change of Hamiltonians and almost complex structure,  we restrict ourselves to $w^+$-monotone symplectic manifolds.  For the convenience of the reader, we recall here the definition of such spaces.
\begin{defi}\label{def w+ mflds} A symplectic manifold $(M,\omega)$ is said to be $w^+$-\textit{monotone} if $\omega(A)>0$ for all $A\in H_2(M)$ in the image of $\pi_2(M)$ under the Hurewicz map satisfying $0 < \langle c_1(M), A \rangle \leq n-2$
\end{defi}
\begin{rem}	Clearly the definition of $w^+$-monotonicity implies that $M$ is weakly monotone in the usual sense of \cite{HS95}.  Examples of $w^+$-monotone manifolds include any symplectic manifold $M$ such that $\dim M \leq 4$, Fano and Calabi-Yau manifolds.
\end{rem}

Let
\begin{equation}\label{eq def twisted loop space}
    \Omega_{\phi} = \{ \gamma : \R \to M \mid \phi\gamma(t+1)=\gamma(t)\}
\end{equation}
be the twisted loop space, let $H: \R\times M \to \R$ be a time dependent Hamiltonian satisfying
\[H_{t+1}=H_t\circ \phi\]
and consider the closed $1$-form $a_H$ defined on it 
\begin{equation*}\label{eq differential action functional}
{a_H}_{\gamma}(\xi) := \int_0^1 \omega(\dot{\gamma}-X_H(\gamma),\xi)dt.
\end{equation*}
We consider (a connected component of) the Novikov cover of $\Omega_{\phi}$ which can be explicitly described as the following quotient.  Let $\gamma_0 \in \Omega_{\phi}$ be a reference path, then:
\begin{equation*}
\widetilde{\Omega_{\phi}}:= \left\lbrace (\gamma,u) \mid u : [0,1]\times [0,1]\to M \ \text{s.t.  } \left\lvert \begin{array}{l}
     \phi u(s,1)= u(s,0)\\ u(0,t) =\gamma_0(t) \\ u(1,t)=\gamma(t)
\end{array}\right.\right\rbrace \bigg/ \raisebox{-.6em}{$\mathlarger{\mathlarger{\mathlarger{\mathlarger{\sim}}}}$}
\end{equation*}
where the equivalence relation is defined as follows: 
 \[ (\gamma,u)\sim (\gamma',u') \Leftrightarrow \gamma=\gamma' \ \text{ and } \int_{[0,1]^2}u^*\omega = \int_{[0,1]^2}{u'}^*\omega. \]
Equivalently, we can characterize $\widetilde{\Omega_{\phi}}$ as the cover associated to the kernel of the evaluation homomorphism 
\[ \text{ev}_{\omega} : \pi_1\left(\Omega_{\phi},\gamma_0\right) \to \R\] 
defined as follows. Let $\left[\omega_{\phi}\right]$ be the class in $H^2(M_{\phi};\R)$ induced by $[\omega]$ using the fact that $\phi$ is a symplectomorphism. Each $1$-cycle $\zeta : S^1 \to \Omega_{\phi}$ representing a homotopy class can be thought as a map $\zeta: S^1\times S^1 \to M_{\phi}$, then by setting
\[\text{ev}_{\omega}(\zeta) :=\langle \left[\omega_{\phi}\right], \zeta_*\left[S^1\times S^1\right]\rangle\]
we have the required map.\\
The differential $1$-form $a_H$ introduced earlier pulls back to an exact form on $\widetilde{\Omega_{\phi}}$ as the differential of the \textit{action functional} $\mathcal{A}_{H,\gamma_0}$:
\begin{align*}
\mathcal{A}_{H,\gamma_0}([\gamma,u]) = \int_{[0,1]^2}u^*\omega + \int_0^1 H\circ \gamma(t) dt
\end{align*}
whose set of critical points $\text{Crit}(\mathcal{A}_{H,\gamma_0})$ can be characterized as the set of elements elements $[(\gamma,u)] \in \widetilde{\Omega_{\phi}}$ satisfying the additional condition:
\begin{equation}\label{eq def of critical points}
\dot{\gamma}(t) =X_H(\gamma(t)).
\end{equation}
Let \[N :=\bigslant{\pi_1\left(\Omega_{\phi},\gamma_0\right)}{\ker \text{ev}_{\omega}}\] 
and consider the induced injective homomorphism induced by the evaluation 
\[ \overline{\text{ev}_{\omega}} : N  \to \R.\]
Following the convention in \cite{HS95},  we define the Novikov field \[ \Lambda_{\omega} := \text{Nov}\left(N, \overline{\text{ev}_{\omega}}, \Z_2 \right)\] as the upward completion of $N$ over $\Z_2$ with respect to the (weight) homomorphism $\overline{\text{ev}_{\omega}}$.\\
\begin{defi}\label{def floer cochain} The Floer cochain complex $CF^k(\phi; \Lambda_{\omega})$ is the $\Z_2$-graded free $\Lambda_{\omega}$-module generated by $\gamma \in \Omega_{\phi}$ satisfying Eq. \ref{eq def of critical points}.  The degree $k$ is given by the sign of $\det\left( 1-d_{\gamma(1)}\phi\circ L_{\gamma}\right)$ where $L_{\gamma} : T_{\gamma(0)}M\to T_{\gamma(1)}M$ is obtained by linearizing Eq. \ref{eq def of critical points} near $\gamma$.  For generic Hamiltonians,  such determinant is non-zero. 
\end{defi}
\begin{rem} For our purposes,  it is convenient to look at $CF^k(\phi; \Lambda_{\omega})$ from a different, although equivalent, perspective: as the $\Z_2$-vector space whose elements are generalised series 
\[ \sum \xi_{[\gamma,u]}\cdot [\gamma,u] \] 
where $\xi_{[\gamma,u]}\in \Z_2$, subject to the condition 
\[\left\lvert \left\lbrace [\gamma,u] \in \text{Crit}^k(\mathcal{A}_{H,\gamma_0}) \text{ s.t.  } \xi_{[\gamma,u]}\neq 0 ,  \mathcal{A}_{H,\gamma_0}([\gamma,u])\geq c \right\rbrace\right\rvert < \infty\] 
for every constant $c > 0$,  where $\text{Crit}^k(\mathcal{A}_{H,\gamma_0})$ is the subset of critical points whose degree is $k\pmod{2}$. A drawback of this point of view is that, in most of the cases, $CF^k(\phi; \Lambda_{\omega})$ is infinitely generated over $\Z_2$.
\end{rem}
 Let $\{J_t\}_t$ be a 1-parameter family of almost complex structures and let $H$ be an autonomous (i.e.  time-independent) Hamiltonian.  Consider the following system of equations
\begin{equation}\label{eq floer symplectomorphism} \begin{cases} &\partial_r u + J_t\partial_t u - \nabla H=0 \\ &\phi \circ u(r,t+1)= u(r,t)\\ &\lim\limits_{r\to \pm \infty} u(r,t)=u_{\pm}(t) \in \text{Crit}(\mathcal{A}_{H,\gamma_0}).\end{cases}
 \end{equation}
It is possible to incorporate the Hamiltonian perturbation inside the almost complex structure in order to obtain an equivalent system of equations by using the Hamiltonian flow of $\psi_t^H$ of the autonomous Hamiltonian $H$. 
\begin{equation}\label{eq floer perturbed symplectomorphism}
 \begin{cases} &\partial_r v + J'_t\partial_t v =0 \\ &\left(\phi \circ\psi_{1}^H\right)\circ v(r,t+1)= v(r,t)\\ &\lim\limits_{r\to \pm \infty} v(r,t)=x_{\pm} \in \text{Fix}(\phi \circ \psi_{1}^H)\end{cases} 
 \end{equation}
 via the identification \[ v(s,t):= (\psi_t^H)^{-1}u(s,t),\]
 where $\{J'_t\}_t$ is a family of almost complex structures that satisfies 
 \[J'_t={\psi_t^H}^*J_t := \left(d{\psi_t^H}\right)^{-1}J_td{\psi_t^H}.\]
 For full details see \cite{DoSa94}[Page 586] and \cite{Oh06}[Section 2.6], keeping in mind the different convention in the definition of the twisted loop space (Eq. \ref{eq def twisted loop space}) which means that we are working with $\phi^{-1}$ in their notation.\\
 To each such solution $u$ we can associate an integer number called the Maslov index $\mu(u)$.  The differentials, similarly to the Morse case,  are given by counting solutions of Eq. \ref{eq floer symplectomorphism} or equivalently Eq. \ref{eq floer perturbed symplectomorphism} with Maslov index $1$.  In order to have a meaningful count,  one has to equip $M$ with regular $1$-parameter families of almost complex structures and Hamiltonians.  In the autonomous case,  additional care is required: following \cite[Theorem 7.3]{HS95}, by taking a small enough autonomous Hamiltonian we can circumvent these transversality issues.  After quotienting out the obvious $\R$-translation action on the solutions,  thanks to $M$ being a $w^+$-monotone manifold we observe that the moduli space of solutions with Maslov index $0$ or $1$ can be compactified by adding broken strips.  This turns out to be essential in order to have $\partial^2=0$.\\ As mentioned at the beginning,  one can prove, using continuation maps,  that different choices of (admissible) Hamiltonians or (regular) almost complex structure gives rise to quasi-isomorphic cochain complexes.
\subsection{The neck-stretching procedure}
In this section we briefly recall the main notions involved in the neck-stretching procedure we are going to use in the proof of Theorem \ref{theorem 1 case 2}.  Classic references for this procedure are \cite{EGH00} (Example 1.3.1 directly applies to our setting) and \cite{BEHWZ03}.  We will closely follow the construction made in \cite[Section 5.2.1]{Eva10}. \\ 

Let $X$ be a contact-type codimension-$1$ submanifold of a symplectic manifold $(M,\omega)$: by definition $X$ comes with a choice of a Liouville vector field $\eta$, defined on a neighborhood $\mathcal{N}$ of $X$,  which is transverse to $X$. The stretching consists in finding a family of compatible almost complex structures on $M$ satisfying certain properties, which we will make precise in a second moment.  Roughly speaking,  the flow of $\eta$ defines an identification \[\mathcal{N}\cong(-\varepsilon,\varepsilon)\times X\] for some $\varepsilon>0$,  and by changing the almost complex structure on $\mathcal{N}$ we change the metric there effectively stretching the neck.  The remaining of this section will be devoted to make precise this idea. \\ Thanks to \cite[Prop. 3.58]{McDS17},  we can equip $X$ with a contact form \[\alpha :=\eta \righthalfcup \omega\] defining a contact structure which is customary to denote by $\xi=\ker \alpha$.  On $\mathcal{N}$ we can write $\omega=d(e^q\alpha)$, where $q$ is the coordinate in the first factor of $(-\varepsilon,\varepsilon)\times X$.  We denote with $\mathbf{R}$ the associated Reeb vector field on $X$,  which we can think of being defined on $\mathcal{N}$ thanks to the identification above.
\begin{defi}\label{def adjusted acs} 
With the notation above,  we say that an almost complex structure $J$ is $\eta$-adjusted (or just adjusted if the choice of Liouville vector field $\eta$ is not ambiguous) if it satisfies the following properties:
\begin{itemize}
\item On $\mathcal{N}$,  $\mathcal{L}_{\eta}J=0$.  In other words $J$ is $\R$-translation invariant, where the translation acts in the direction of $\eta$
\item $J(\eta)=\mathbf{R}$ and $\left.J\right|_{\xi}$ is a $d\alpha$-compatible almost complex structure on $\xi$.
\end{itemize}
\end{defi}
Note that an adjusted almost complex structure $J$ is automatically $d(e^q\alpha)$-compatible on $\mathcal{N}$. \\ Let $\mathcal{F}_s$ be the flow of $\eta$ at time $s$. Let $\nu\geq 0$ be any real number, using this flow we can \textit{stretch} the neighborhood $\mathcal{N}$ of $M$ with parameter $\nu$ as follows: let 
\begin{align*}
\Phi_{\nu} : I_{\nu} \times X &\to M\\
(s,m) &\mapsto \mathcal{F}_{\beta_{\nu}(s)}(m)
\end{align*}
where $I_{\nu} := (-\nu-\varepsilon,\nu+\varepsilon)$ and $\beta_{\nu} : I_{\nu} \to (-\varepsilon, \varepsilon)$ is a smooth strictly monotone increasing function that satisfies the following constraints: 
\[\beta_{\nu}(s) = \begin{cases} s+ \nu &\mbox{ if } s \in (-\nu-\varepsilon, -\nu-2/3\varepsilon] \\ \frac{\varepsilon}{2\nu + \varepsilon}s &\mbox{ if } s \in [-\nu -\varepsilon/2, \nu + \varepsilon/2] \\ s-\nu &\mbox{ if } s \in [\nu + 2/3\varepsilon, \nu + \varepsilon). \end{cases}\]
See Fig. \ref{fig:beta} for a possible candidate of $\beta_{\nu}$.
\begin{figure}
\begin{center}
\includegraphics[scale=1.5]{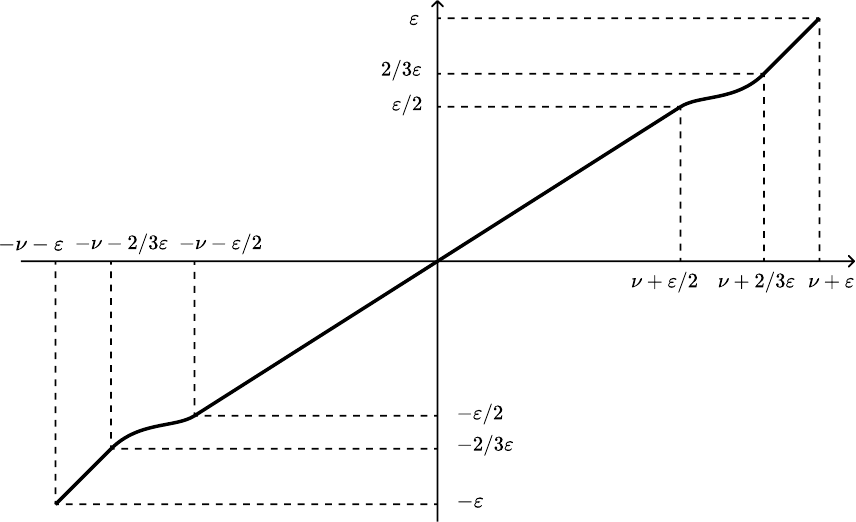}
\captionof{figure}{a possible candidate for $\beta_{\nu}$} \label{fig:beta}
\end{center}
\end{figure}

Let $\widetilde{J}^{\nu}$ be a $\eta$-adjusted almost complex structure on $I_{\nu}\times X$ such that it coincides with $J$ on the distribution $\xi$.  We reconstruct the manifold $M$ but with a new $\omega$-compatible almost complex structure as follows:
 \begin{equation*}
\left(M, J^{\nu}\right):= \left( M\setminus [-2/3\varepsilon,2/3\varepsilon]\times X, J\right)\bigsqcup_{\Phi_{\nu}} \left( I_{\nu}\times X ,  \widetilde{J}^{\nu}\right).
\end{equation*}
On the common overlap of the two pieces the differential of the map $\Phi_{\nu}$ is the identity.  Hence,  we can glue the two almost complex structures to form a global one which we denote with $J^{\nu}$.  The sequence $\{J^{\nu}\}_{\nu}$ of almost complex structures on $M$ is called the \textit{neck-stretch} of $J$ along $X$.  
\begin{rem}\label{rem exact form on neck} If we equip $I_{\nu} \times X$ with the symplectic form \[\omega_{\beta_{\nu}}:=d\left( e^{\beta_{\nu}(s)}\alpha\right)=\Phi_{\nu}^*\omega\] the map $\Phi_{\nu}$ becomes a symplectic embedding by definition.  Therefore we are deforming the almost complex structure $J \rightsquigarrow J^{\nu}$ on $\mathcal{N}$ while keeping the symplectic structure fixed, implying that we are changing the associated metric on $\mathcal{N}$ as well.  With the new induced metric the map $\Phi_{\nu}$ becomes an isometry on its image.
\end{rem}
Let us conclude this section by working out some concrete computations in the case that $M$ is a Riemann surface: thanks to our specific choice of coordinates, in the higher dimensional case we are interested in, the remaining $n-2$ coordinates will not interfere in the following observations.
\begin{ex}\label{ex stretched metric surface}
Let $V$ be a simple closed curve in a Riemann surface $\Sigma$.  By choosing an embedding $S^1 \to V$ we have a coordinate $z$ that describes all the points in $V$.  The contact form $dz$ on $TV$ satisfies the (trivial) identity 
\[ d(dz)=0=\left.\text{vol}_{\Sigma}\right|_{V},\] 
hence by \cite[Prop 3.58]{McDS17} we can find a Liouville vector field $\eta$ defined on a neighborhood $\mathcal{N} \subset \Sigma$ of $V$ transverse to $V$.  We can use its flow to construct the following trivialization: $\mathcal{N} \cong (-\varepsilon, \varepsilon)\times S^1$,  which let us assign to any given point $p \in \mathcal{N}$ coordinates $(q,z)$. \\ Now consider $(-\nu-\varepsilon, \nu+\varepsilon) \times S^1$ equipped with the symplectic form $\omega = d(e^{\beta_{\nu}(s)}dz)=e^{\beta_{\nu}(s)}\beta_{\nu}'(s)ds \wedge dz$,  where $dz$ is the canonical contact form of $S^1$, $z$ representing the angular coordinate on the circle.  As before,  we will use the letter $s$ to denote the coordinate in the Liouville direction in $I_{\nu}\times S^1$. \\
On $(-\nu-\varepsilon, \nu+\varepsilon) \times S^1$ we have the canonical a.c.s \[ \widetilde{J}^{\nu}_{(s,z)}= \begin{pmatrix}
0 & -1 \\ 1 & 0
\end{pmatrix}.\]
Let us now consider $({\Phi_{\nu}})_*J$ on $\mathcal{N}$,  and let $a,b \in \R$ be the components of a given tangent vector with respect to the basis given by $(\eta,\mathbf{R})$,  the Liouville and Reeb vector fields in $\mathcal{N}$.  Similarly to \cite[Lemma 5.2.8]{Eva10} we have,  for a vector $v=a\eta + b\mathbf{R}$ at the point $(q,z)$:
\[{\Phi_{\nu}}_* \widetilde{J}^{\nu}v = -\beta_{\nu}'\left(\beta_{\nu}^{-1}(q) \right)b\eta + \dfrac{1}{\beta_{\nu}'\left(\beta_{\nu}^{-1}(q)\right)}a\mathbf{R}.\]
Notice that $({\Phi_{\nu}})_*J$ is not adjusted to the sphere bundle of $V$,  for the Liouville field $\eta$ anymore,  but it's still cylindrical and still interchanges the \textit{directions} given by $\eta$ and $\mathbf{R}$.  This observation applies in any dimension.\\
Thanks to our choice of coordinates $(q,z)$,  the symplectic form can be written as $e^qdq\wedge dz$.  Then the associated metric becomes
\[ g_{\nu}(v,v) = \omega(v,(\Phi_{\nu})_* \widetilde{J}^{\nu}v)= e^q\left(a^2\dfrac{1}{\beta_{\nu}'\left(\beta_{\nu}^{-1}(q)\right)}+ b^2\beta_{\nu}'\left(\beta_{\nu}^{-1}(q) \right)\right),\] 
for the general case compare with \cite[Lemma 5.2.8]{Eva10}.  The last computation we want to spell out in details is an estimate of the geodesic length of the neck $\mathcal{N}$ with respect to the metric $g_{\nu}$.  More precisely,  we are interested in finding an equation of the geodesic of the form $\gamma_{\nu}(r)=(q(r),0)$ lying in $[ -\varepsilon/2,  \varepsilon/2]\times V$.  The geodesic equation boils down to 
\[ \dfrac{d^2q}{dr^2}+\Gamma^1_{1,1}\left(\dfrac{dq}{dr}\right)^2=0\] 
and thanks to the specific value of $\beta_{\nu}$ the Christoffel symbol (w.r.t.  $g_{\nu}$) can be easily computed: 
\[\Gamma^1_{1,1}= -\dfrac{1}{2}.\]
Therefore,  the geodesic is given by the following parametrization:
 \[ \gamma_{\nu}(r)=\left(\log(r^2),z_0\right)\]
 for a given $z_0\in V$.  For a given parameter $\nu$,  the speed of the geodesic must be constant and a quick computation using the metric $g_{\nu}$ reveals that 
 \[ \| \dot{\gamma_{\nu}}\|_{\nu} =2\sqrt{\dfrac{2\nu+\varepsilon}{\varepsilon}}.\]
A similar computation reveals that a solution to the same geodesic equation but for ${q \in [-\varepsilon, -2/3\varepsilon]\cup [2/3\varepsilon,\varepsilon]}$ is given by the same function $\gamma_{\nu}(r)=\left(\log(r^2),z_0\right)$.  The magnitude of its speed is easily seen to be equal to $2$.\\ Let $q\in [-\varepsilon/2,\varepsilon/2]$,  notice that for $\nu\to \infty$  \[ \| \dot{\gamma_{\nu}}\|_{\nu}\to +\infty.\] Moreover it is immediate to see that $\gamma_{\nu}$ maps the segment $\left[e^{-\varepsilon/4},e^{\varepsilon/4}\right]$ to $[-\varepsilon/2,\varepsilon/2]$.  Since the domain is independent from $\nu$,  we conclude that for $\nu\to \infty$, the geodesic length of the neck is growing to infinity. 
\end{ex}
\begin{rem}\label{rem arclength}
We will need the computation about geodesic later in the paper so let us assume that our curves are parametrised with their arc-length to ensure they have constant unit speed.
\end{rem}

\subsection{Perturbing the Cauchy-Riemann-Floer equations using Morse functions}
For each $l \in \{1, \dots , m\}$,  by the Darboux - Weinstein theorem there exists $\delta>0$ such that we have a symplectic embedding \[ T^*_{\delta}(S^n)_l \hookrightarrow M\] that identifies the zero section $S^n$ with the Lagrangian sphere $V_l$.  The precise quantity $\delta$ might depend on the index $l$, but this is of little importance here so for simplicity we suppress the "$l$" from the notation. Since the Lagrangian spheres are disjoint from each other, by taking smaller $\delta$ if necessary, we can always assume that the images of these symplectic embedding are disjoint from each other.\\ We remind the reader that $T^*_{\delta}(S^n)_l$ is the disk bundle of the cotangent bundle of $S^n$ whose vectors have norm less or equal to $\delta$ with respect to the metric induced by the canonical symplectic form and almost complex structure.  The index $l$ keeps track of which Lagrangian sphere $V_l$ we are working on.\\ 
Let us fix such an embedding and equip $M$ with an $\omega$-compatible almost complex structure $J$ that extends the canonical ones induced on $\bigcup_l T^*_{\delta}(S^n)_l$.  This implies we are choosing a framing for the Lagrangian sphere $V_l$,  following the convention in \cite{Sei03} we suppress the specific choice of the framing in the notation,  understood that it has been fixed once and for all.\\ Such an identification gives us preferred choices for Liouville vector fields defined around each Lagrangian spheres,  which we will denote with $\eta_l$: specifically the push-forward along the chosen embedding of the radial vector field defined on $T^*_{\delta}(S^n)_l$.  In this way,  any (rescaled) sphere bundle $S(S^n)_l$ lying in $T^*_{\delta}(S^n)_l$ inherits the structure of a codimension $1$ submanifold of $T^*_{\delta}(S^n)_l$, hence of $M$.  We pick one and denote it with $S(V_l)$,  we choose $\varepsilon >0$ such that, for each Lagrangian sphere $V_l$,  we have an isometric embedding 
\[ (-\varepsilon,\varepsilon)\times S(V_l) \subset T^*_{\delta}(S^n)_l \subset M. \]
Notice that we are using the flow of the Liouville vector field $\eta_l$ to define the coordinate $q\in (-\varepsilon,\varepsilon)$.  Without loss of generality,  the complement of $(-\varepsilon,\varepsilon)\times S(V_l)$ inside $T^*_{\delta}(S^n)_l$,  denoted with $D(V_l)$,  is isomorphic to the disk bundle $T^*_{\delta'}(S^n)_l$,  for some $\delta'>0$.   The exact value of $\delta'$ does not matter. \\
To ease the notation, let us set $M':=M\setminus \bigcup_l^m T^*_{<\delta}(S^n)_l$.  $M'$ is a compact symplectic manifold with $m$ boundary components $\{\partial M'_l\}_{l=1}^m$.  For each of this boundary components we keep track of the sign $\sigma_{l}$ of the Dehn twist about $V_l$.  
\begin{figure}[H]
\begin{center}
\includegraphics[scale=.5]{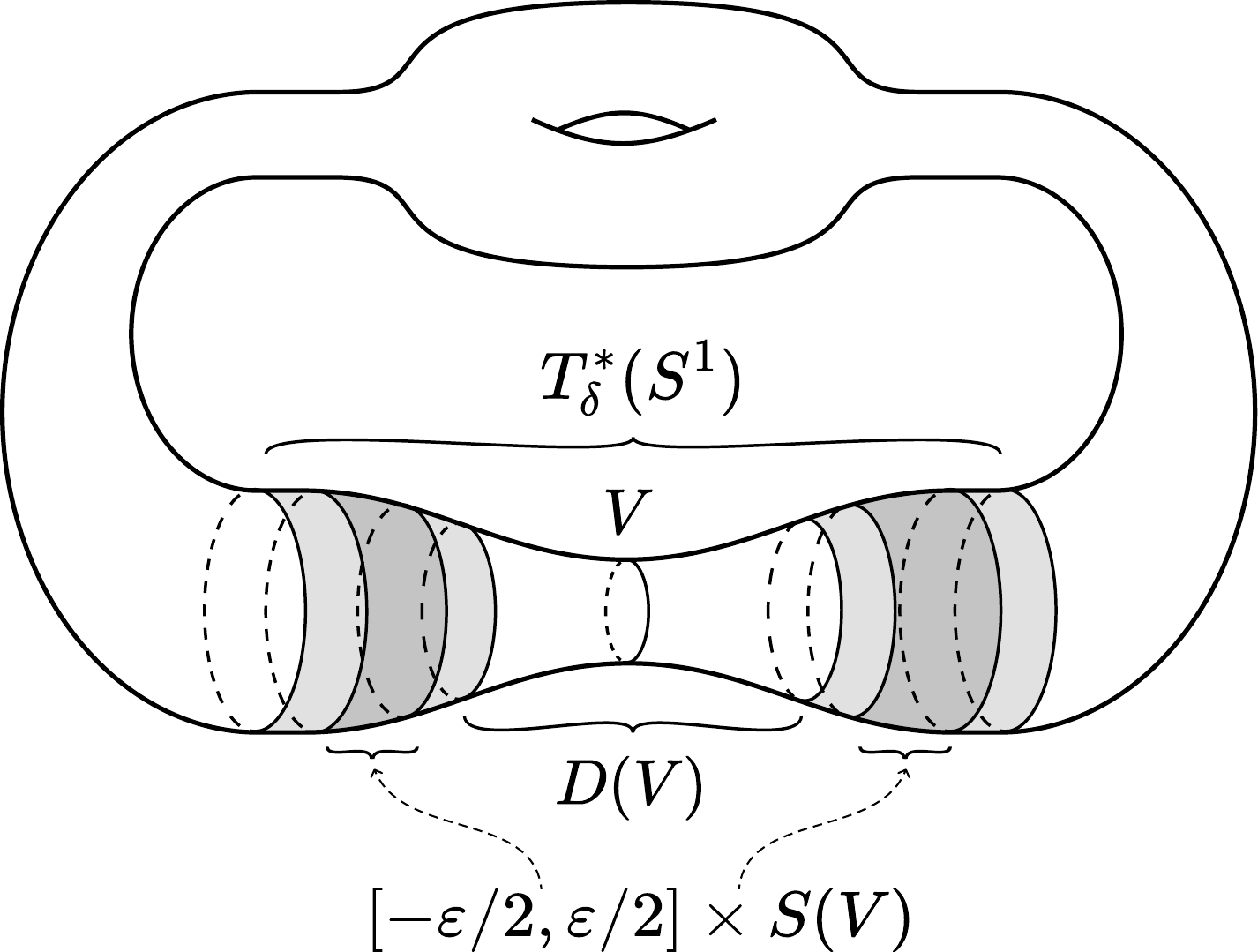}
\captionof{figure}{The setup for the neck-stretch in dimension $2$ for a single Lagrangian $V$} \label{fig:setup}
\end{center}
\end{figure}
Notice that, for each index $l$, we can consider the (radial) geodesic vector field $\dot{\gamma}_{\nu,l}(r)$ on a sufficiently small neighbourhood of $\partial M'_l$. This is due to the fact that, for each $l$, 
\[T^*_{\delta}(S^n)_l\cap M'\neq \emptyset.\]
The specific choice of parameter $\nu$, which a priori affects the metric according to which the geodesics are taken, is irrelevant: thanks to the properties of the function $\beta_{\nu}$, for different choice of neck-stretch parameter the metric is constant in a neighbourhood of $\partial M'_l$.

\begin{prop} \label{prop existence of Morse functions} Let $\rho>0$ be an arbitrary small real number.  Let $h': M' \to \R$ be an admissable Morse function (in the sense of \cite[Def. 7.1]{FHS95})  whose $C^2$-norm is bounded above by $\rho$ and whose gradient flow on $M'\cap T^*_{<\delta}(S^n)_l$ coincides with $\sigma_l\rho\dot{\gamma}_{\nu,l}(r)$ for some parameter $\nu$. We can find a family of admissable Hamiltonians $h^{\nu} : M \to \R$ such that,  for each $\nu > 0$ the following conditions are satisfied:
\begin{enumerate}
\item $\left. h^{\nu} \right|_{M'}\equiv h'$
\item Let $\nabla_{\nu}h^{\nu}$ be the gradient of $h^{\nu}$ with respect to the metric $g^{\nu}$. If $\sigma_l=+1$ (resp. $\sigma_l=-1$), on $T^*_{\delta}(S^n)_l\setminus V_l$ the gradient $\nabla_{\nu}h^{\nu}$ coincides with a positive (resp. negative) multiple of the Liouville vector field $\eta_l$. Moreover $\| h^{\nu}\|_{C^2}\leq \rho$.
\end{enumerate}
\end{prop}
Since $S(V_l)$ is a contact-type submanifold of $M$, the Liouville vector field $\eta_l$ is transverse to it, which implies that the same is true for $\sigma_l\rho\dot{\gamma}_{\nu,l}(r)$. We are following the \textit{cohomology} convention of \cite[Chapter 4.1]{AuDa14} for outward/inward vector field in order to set up relative Morse cohomology.
\begin{proof}
We will prove the proposition for a single (positive) Dehn twist $\tau_V$. The general result will be obtained by iterating the proof for each Lagrangian sphere starting with a Morse function on $M'$ that satisfies the required conditions on each boundary component.\\ We begin by defining what the gradient should look like on the neck $(-\varepsilon,\varepsilon)\times S(V)$: following the computations done in Example \ref{ex stretched metric surface} and their obvious generalization in higher dimension, we impose the gradient of our candidate function to be equal to $\rho\dot{\gamma_{\nu}}(r)$ (we remind the reader of Remark \ref{rem arclength}).  Now since the Liouville vector field $\eta$ is defined on $ T^*_{\delta}(S^n)$,  hence in particular on $D(V)$, we can assume that after an appropriate smooth rescaling,  our candidate gradient is defined on the entire $T^*_{\delta}(S^n)$.  Without loss of generality,  the gradient of $h'$ on the boundary of $M'$ coincides with $\rho\dot{\gamma_{\nu}}(r)$.  Note that since the metric nearby $\{\pm\varepsilon\}\times S(V)$ is independent from the parameter $\nu$, we can now integrate the candidate gradient we constructed before and extend $h'$ to an Hamiltonian function $h^{\nu}$ on the entire $M$. It is immediate to see that properties 1. and 2. are then satisfied.
\end{proof}

\begin{lemma}\label{lem Hamiltonian flow of morse function} Let us denote by $\psi^{\nu}_t : \ M \times \R \to M$ the Hamiltonian flow of $h^{\nu}$ with respect to the symplectic form $\omega$. Then, depending on $\sigma_l$,  $\psi^{\nu}_t$ acts as a constant translation in the positive or negative Reeb direction in a neighbourhood of each Lagrangian sphere $V_l$
\end{lemma}
\begin{proof}
Remember that 
\[\Phi_{\nu} : I_{\nu} \times S(V_l) \to M\] 
is an isometric embedding if we equip $M$ with the metric induced by $\omega$ and $J^{\nu}$ since this identification sends the canonical $\eta$-adjusted almost complex structure on $I_{\nu} \times S(V_l)$ to $J^{\nu}$.  Thanks to Prop. \ref{prop existence of Morse functions} 
\[ \nabla_{\nu}h^{\nu}=  \sigma_l \rho\dot{\gamma_{\nu}}(r),\] 
hence if we pull it back via $\Phi_{\nu}$, we get \[ \nabla_{\nu}h^{\nu}= \sigma_l \rho\partial_s\]
The Hamiltonian vector field is easily then computed  as \[\left.X_{\nu}\right|_{I_{\nu} \times S(V_l)}= \sigma_l \rho\mathbf{R} \ \text{ on } I_{\nu} \times S(V_l)\] and its flow is clearly a constant translation in the Reeb direction.  Full details on the identifications can be found in  \cite[Proposition 1.21 and Lemma 1.30]{Pat99}.
\end{proof}

\begin{defi} Consider the composition of Dehn twists $\tau : M \to M$ around the disjoint Lagrangian spheres.  Following the convention of Eq.  \ref{eq floer perturbed symplectomorphism}
we define the \textit{perturbed} composition of Dehn twists as
\begin{equation*}
    \tau_p^{\nu}:=\tau \circ \psi_{1}^{\nu}.
\end{equation*}
\end{defi}

\begin{rem}\label{rem sign morse function}
For $\sigma_l=\pm 1$,  the perturbed Dehn twists $\tau_p^{\nu}$ does not have fixed points along the stretched necks and in general nearby each of the $V_l$.  In the case of a single Lagrangian sphere and  $\sigma_l=1$,  Lemma \ref{lem Hamiltonian flow of morse function} shows that the perturbation moves points in the positive Reeb direction and the Dehn twist moves points along the (positive) Reeb direction since it coincides with the rescaled geodesic flow on the cotangent bundle $T^*_{\delta}(S^n)_l$ (see \cite[Theorem 6.1]{Etn06}).  Following the same reasoning as in \cite[Lemma 20.1]{Sei01},  we observe that the perturbed Dehn twist does not have any fixed point on $T^*_{\delta}(S^n)_l$ and any other fixed point arises as a fixed point of $\psi_{1}^{\nu}$. The same reasoning works in the case of $\sigma_l=-1$,  since we are moving everything in the opposite Reeb direction. 
\end{rem}
 In the case $\nu=0$, i.e.  $M$ equipped with the original almost complex structure $J$,  we will denote with $\tau_p$ the perturbed composition of Dehn twists. 
 \subsection{Existence of suitable regular adjusted almost complex structures}
Let us work in the case of a single Dehn twist $\tau=\tau_V$, the general case is done in the exactly same way since the supports of the Dehn twists are disjoint.\\
Assume that the support of the Dehn twist $\tau_{V}$ is contained in the smaller neighborhood $D(V)$ defined previously.  By working in the vertical tangent bundle of $M_{\tau_V}$,  we can find a family of $\omega$-compatible almost complex structures $\{\widetilde{J}_t\}$ such that:
\begin{enumerate}
\item $\tau_V^*\widetilde{J}_{t+1}=\widetilde{J}_t$
\item On $(-\varepsilon,\varepsilon)\times S(V)$,  for each pair of indexes $t,t'$, we have $\widetilde{J}_t=\widetilde{J}_{t'}=:\widetilde{J}$ 
\item On $(-\varepsilon,\varepsilon)\times S(V)$, $\widetilde{J}$ is $\eta$-adjusted,  where $\eta$ is the (radial) Liouville vector field defined in the neighborhood of the Lagrangian sphere $V$.  
\end{enumerate}
Since the support of the Dehn twist is strictly contained in $D(V)$,  requests 1. and 2. are not contradicting each other. By perturbing the family $\{\widetilde{J}_t\}$ outside the subspace swept by $T_{\delta}^*(S^n)$,  there exists a family of regular almost complex structures $\{J'_t\}$ satisfying \cite[Theorem 9.1]{Sei97} and both condition 1. and 2.: the absence of fixed points for our perturbed Dehn twist in the neighborhood of $V$ implies that every (finite energy) solution of the above equation will eventually exit such neighborhood. The same is true for any (non-constant) pseudo-holomorphic sphere, due to the exactness of the symplectic form on the Weinstein neighbourhood $D(V)$. Therefore thanks to \cite[Remark 3.2.3]{McDS12} we can achieve regularity by perturbing our compatible a.c.s. on the complement of $T_{\delta}^*(S^n)$. We conclude this discussion by reminding the reader that a similar perturbation was used in \cite[Theorem 3.4]{OS04}.\\ Let 
\[h^0 : M \to \R\] 
be the Hamiltonian function constructed in Prop. \ref{prop existence of Morse functions} for $\nu=0$, we can now define 
\[J_t:=\left(\psi_t^{0}\right)^*J'_t. \] 
Notice that the family $\{J_t\}$ satisfies the relation 
\begin{equation}\label{eq acs is periodic}
    {\tau_p}^*J_{t+1}=J_t
\end{equation} 
and it is still $\eta$-adjusted on $(-\varepsilon,\varepsilon)\times S(V)$: by Lemma \ref{lem Hamiltonian flow of morse function} ${\psi_t^{0}}$ represents a constant translation along the Reeb orbits and hence preserves $J'_t$.\\

For the stretched almost complex structures, let $\{\nu_i\}_{i\in \N}$ be an increasing sequence of positive numbers such that:
\begin{itemize}
    \item $\nu_0 = 0$ \item $\lim\limits_{i\to +\infty}\nu_i =+\infty$
\end{itemize}
and let $\{\widetilde{J}^{\nu_i}_t\}$ be the associated family of a.c.s. realizing the neck-stretching. Notice that properties 1. and 2. are still satisfied by it. For each parameter $\nu_i$, we then repeat the procedure above starting from $\{\widetilde{J}^{\nu_i}_t\}$, obtaining a \textit{regular} almost complex structure $\{J_t^{\nu_i}\}$ such that 2. and Eq. \ref{eq acs is periodic} are still valid.

\begin{rem}\label{rem induced metric on neck}
    For any neck-stretch parameter $\nu_i$, each $\{J_t^{\nu_i}\}$ induces a compatible metric 
    \[g^{\nu_i}_t(-,-):=\omega(J^{\nu_i}_t - , - )\]
    on the mapping torus $M_{\tau^{\nu_i}_p}$, or - \textit{equivalently} - a family of metrics on $M$ satisfying an obvious periodicity condition. Thanks to property 2., which holds for any $\nu_i$, on $(-\varepsilon,\varepsilon)\times S(V)$ the metric is independent on $t$:
    \[g^{\nu_i}_t(-,-)=\omega(J^{\nu_i}_t - , - )= \omega(J^{\nu_i}_{t'} - , - ) =: g^{\nu_i}(-,-) \]
    so we can safely drop the dependency on $t$ as long as we are working on the neck of $V$. This will be important later for establishing certain estimates on the area of the curves.
\end{rem}


\section{Lower bounds on the energy of certain pseudoholomorphic curves}
Let us begin by setting up some notation: let $\mathcal{M}^{\nu_i}_k(x_-,x_+)$ be the moduli space of $J^{\nu_i}$-holomorphic maps $u^i : \R \times \R \to M$ satisfying the Floer equations with Maslov index $k$.
\begin{equation}\label{eq floer symplectomorphism stretched}
 \begin{cases} &\partial_r u^i + J_t^{\nu_i}\partial_t u^i =0 \\ & \tau_p^{\nu_i}u^i(r,t+1)= u^i(r,t)\\ &\lim\limits_{r\to \pm \infty} u^i(r,t)=x_{\pm} \in \text{Fix}(\tau_p^{\nu_i}).\end{cases} 
 \end{equation}
\begin{defi} Let $u^i \in \mathcal{M}^{\nu_i}_k(x_-,x_+)$, using the Cauchy-Riemann equations, we can define the \textit{energy} of $u^i$ as 
\[ E(u^i) :=\int_{\R \times [0,1]} {u^i}^*\omega.\]
\end{defi}
\begin{rem} The $t$-dependent almost complex structure $\{J^{\nu_i}_t\}$ induces a family of compatible metrics $\{g^{\nu_i}_t\}$ on $M$.  With that,  the energy of $u^i$ satisfies
\[E(u^i)= \int_{\R \times [0,1]}\left(\lvert du^i(r,s)\rvert_{g^{\nu_i}_s}^{op}\right)^2dsdr, \]
where $\lvert \hspace*{1mm}\cdot \hspace*{1mm} \rvert_{g^{\nu_i}_t}^{op}$ is the operator norm for linear maps induced by the family of metrics $\{g^{\nu_i}_t\}$ on $M$.\\
\noindent Thanks to Rem.  \ref{rem induced metric on neck},  on $(-\varepsilon,\varepsilon)\times S(V)$ the family of metrics is constant and coincides with $g^{\nu_i}$,  which implies that on the neck,  the energy of a $J^{\nu_i}$-holomorphic map coincides with its area.
\end{rem}
This entire section is devoted to proving the following:
\begin{theorem}\label{thm lower bound on energy for strips} Let $\{\nu^i\}_i$ be a sequence of neck-stretching parameters. Up to passing to a subsequence of them, which we will still denote with $\{\nu^i\}_i$, we can find a sequence of positive real numbers $\{\lambda_i \}_i$ satisfying \[ \lim_{i\to+\infty} \lambda_i= +\infty\] with the following property: for each index $i$, for any pair of fixed points $(x_-,x_+)\in \text{Fix}(\tau_p^{\nu_i})$, every solution $u^i$ of Eq. \ref{eq floer symplectomorphism stretched} intersecting $\{0\}\times S(V_l)$, has its energy bounded below by $\lambda_i$.
\end{theorem}
\begin{rem} Our perturbed Dehn twists do not have fixed points in the neighbourhood of $V_l$ bounded by $\{0\}\times S(V_l)$, hence if a solution of Eq.  \ref{eq floer symplectomorphism stretched} goes through any of the $V_l's$ it must intersect $\{0\}\times S(V_l)$ as well.
\end{rem}
For simplicity,  we assume now that there is only one Lagrangian sphere $V$ together with the associated Dehn twist $\tau_V^{\sigma}$ and that there is only one pair of fixed points $x_-,x_+\in M$. By repeating the proof for each $V_l$ and for each pair of fixed points and taking the minimum among the lower-bounds found,  we obtain a global lower estimate on the energy and the theorem will be proved. \\ 
\begin{proof}
Let $(x_-,x_+)\in \text{Fix}(\tau_p)$ be our chosen pair of non-degenerate fixed points. For each neck-stretch parameter $\nu_i$, let $U_i:=\{u^i_q\}_{q}$ be a family of $J^{\nu_i}$-holomorphic curves, solution to Eq. \ref{eq floer symplectomorphism stretched}, connecting $x_-$ to $x_+$ and passing through $\{0\}\times S(V)$.\\ 
Let 
\begin{equation*}\label{eq definition set energy differences}
    ED_i:= \{ E \in \R \mid \exists \ u \in U_i \text{ s.t } E=E(u) \}
\end{equation*}
be the set of the energies of curves passing through the neck and connecting the two chosen fixed points. One can think of $ED_i$ as the set of differences \[E(u^i_q)=\mathcal{A}_{h^{\nu_i}}\left(\widetilde{x_+}^i_q\right)-\mathcal{A}_{h^{\nu_i}}\left(\widetilde{x_-}\right),\] where $\widetilde{x_-}$ is a chosen lift of $x_-$ and $\widetilde{x_+}^i_q$ is the induced lift of $x_+$ by the lift of $u^i_q \in U_i$ starting at $\widetilde{x_-}$.\\
Clearly each $ED_i$ is a subset of the real numbers (strictly) bounded below by $0$, hence it admits a non-negative infimum 
\[\overline{E}_i:=\inf ED_i\geq 0.\] 
We have two possible cases:
\begin{enumerate}
    \item $\left[\lim_{i \to \infty}\overline{E}_i=\infty\right]$ In this case there is nothing to prove. We can find a subsequence of neck-stretch parameters, which we will still denote by $\{\nu_i\}_i$, such that the energy required to go through the neck increases monotonically to $+\infty$. Therefore we can simply set
    \[ \lambda_i:=\overline{E}_i.\]
    \item $\left[ \exists \Upsilon \in \R \text{ s.t. } \overline{E}_i\leq \Upsilon \right]$ By the definition of infimum, after possibly choosing a slightly bigger constant which we will still denote by $\Upsilon$ for convenience, for each neck-stretch parameter $i$ we can find $u^i\in U_i$, such that $E(u^i)\leq \Upsilon$. From now on $\{u^i\}_i$ will denote a sequence of curves $u^i \in U_i$ satisfying \[\forall i, \ E(u^i) \leq \Upsilon.\]
    Let $\{u^i\}_i$ be our sequence of solutions to Eq. \ref{eq floer symplectomorphism stretched} with bounded energy constructed above. For each $t \in [0,1]$ we define
        \begin{align*}
            L_i(t) &:= \left\lbrace r \in \R \mid u^i(r,t) \in \left[-\varepsilon/2, \varepsilon/2 \right]\times S(V)\right\rbrace
        \end{align*}
    and we denote with $L_i$ their union
    \begin{align*}
        L_i &:= \bigcup_{t\in [0,1]} L_i(t)\times \{t\} \subset \R \times [0,1].
    \end{align*}
    It is immediate to see that for each $t \in [0,1]$ and for each $i$,  the set $L_i(t)$ is measurable with respect to the Lebesgue measure $\mu$ on $\R$.
    \begin{figure}[H]
    \begin{center}
    \includegraphics[scale=.7]{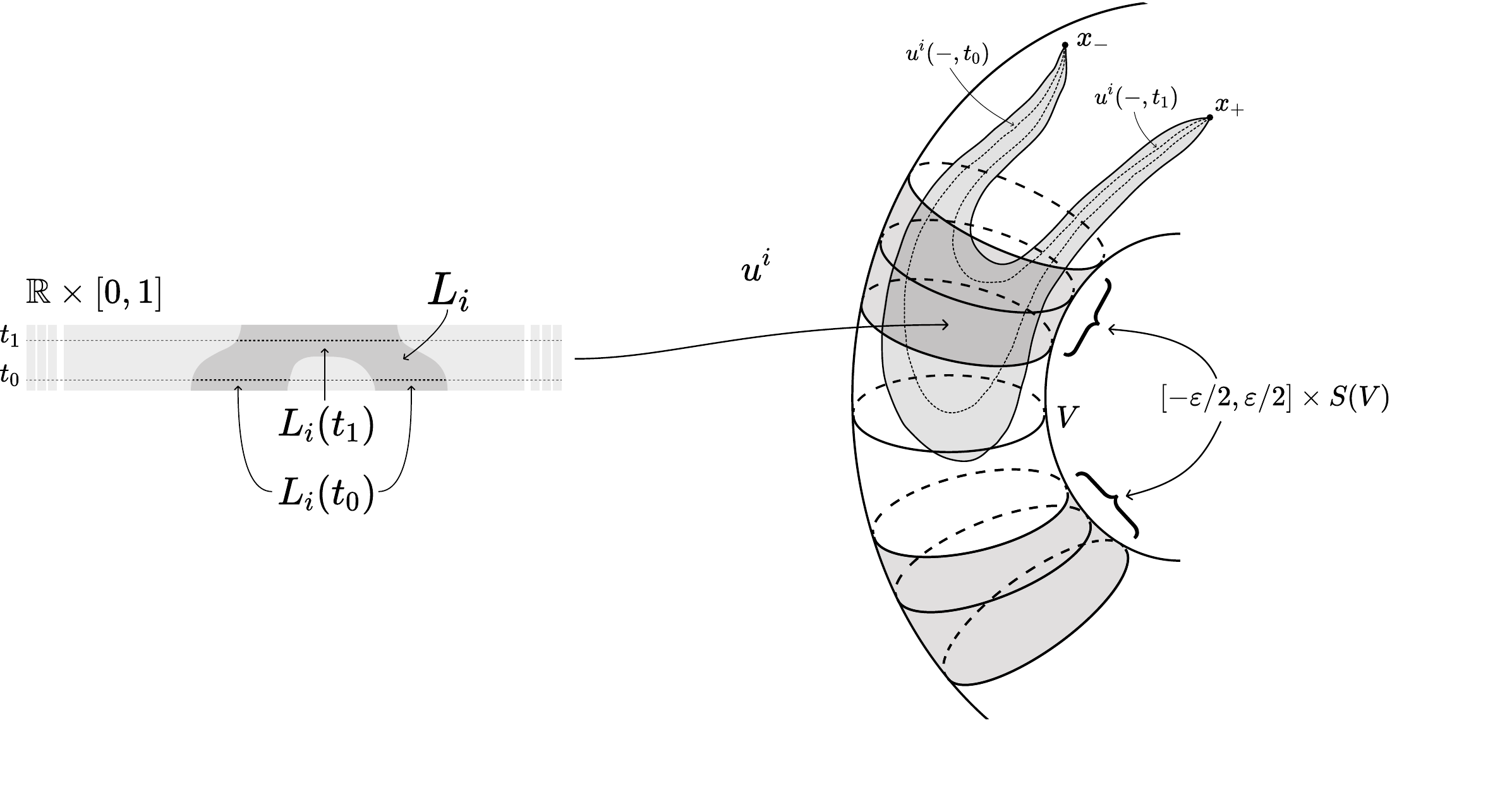}
    \captionof{figure}{The sets $L_i(t_0),L_i(t_1)$ and $L_i$ associated to the curve $u^i$} \label{fig:figLi}
    \end{center}
    \end{figure}
    
    \begin{prop} There exists a constant $C>0$ such that, $\forall i \in \N$, we have a uniform bound on the $L^{\infty}$-norm of the differentials $\{du^i\}_i$ over $L_i$. In other words,
    \begin{equation*}
     \| \restr{du^i}{L_i}\|^{\infty}_{g^{\nu_i}_t}:=\hspace{0.1cm}\sup_{(r,t)\in L_i}\lvert du^i(r,t)\rvert_{g^{\nu_i}_t}^{op} \leq C.
    \end{equation*}
    \end{prop}
    \begin{proof}
        We proceed by contradiction. Let us assume that there exists a sequence of points $\{z_i\}_i$, $z_i:=(r_i,t_i) \in L_i$ such that 
        \[\lim_{i\to \infty}\hspace{0.1cm} \lvert du^i(z_i)\rvert_{g^{\nu_i}_t}^{op} =+\infty,\]
        and let us denote with $C_i$ the value $\lvert du^i(z_i)\rvert_{g^{\nu_i}_t}^{op}\in \R_{\geq 0}$. After possibly passing to a subsequence we can assume that we have a monotone sequence of non-negative real numbers $\{C_i\}_i$ such that
    \begin{equation*}
        \lim_{i \to \infty} C_i=+\infty.
    \end{equation*} 
    We want to apply the bubbling analysis explained in \cite[Lemma 5.11]{BEHWZ03} for maps into a family of stretched manifolds. Notice that here we are dealing with strips (which have \textit{boundary}), a case not considered in \cite{BEHWZ03}. Thanks to the periodicity condition on the solutions of Eq. \ref{eq floer symplectomorphism stretched}, we can think of them as maps whose domain is $\R\times \R$ (as in \cite{DoSa94} or \cite{Sei96}). Moreover, the cylindrical interpretation of solutions to Eq. \ref{eq floer symplectomorphism stretched} as sections of $\R \times M_{\tau_p^{\nu_i}}$ should further convince the reader that there are no problems with the boundary of the strips.\\
    If we define
    \[d_i := d^{\nu_i}\left(u^i(z_i),M'\cup D(V)\right) \]
    then
    \[\lim_{i\to \infty} d_i=+\infty.\]
    We apply then \cite[Lemma 5.11]{BEHWZ03} to the family of functions
    \[\restr{u^i(z+z_i)}{B_{\min(1,d_i/C_i)}(0)} : B_{\min(1,d_i/C_i)}(0)\to \left([-\varepsilon,\varepsilon]\times S(V),J^{\nu_i}\right)\footnote{In the case $\lim_i d_i/C_i=0$, in the notation of the aforementioned lemma, we take $\delta_n :=d_n/C_n$ and the proof follows.}\]
    and get a \textit{non-constant} holomorphic map 
    \begin{equation*}
        \widetilde{u} : \C \to \R \times S(V) \text{ s.t. } E(\widetilde{u})\leq \Upsilon.
    \end{equation*}
    Since $\R \times S(V)$ with the induced symplectic form is the symplectisation of a closed contact ma\-ni\-fold, by the removable singularity theorem by J.-C. Sikorav (\cite[Theorem 4.5.1]{Aud94} together with example 4 pg. 179 of the same reference), we obtain a \textit{non-constant} pseudo-holomorphic sphere (still denoted with $\widetilde{u}$)
    \begin{equation*}
        \widetilde{u} : S^2 \to \R \times S(V).
    \end{equation*}
    On the other hand, $\widetilde{u}$ has to be constant due to the exactness of the symplectic form on the neck. We have then found the required contradiction and therefore we can conclude that there cannot be such a sequence of points $\{z_i\}_i$ over which the norm of the differentials keeps increasing.
    \end{proof}
    The next lemma provide us a crucial estimate on the distance of certain pairs of points in the neck of $S(V)$:
    \begin{lemma}\label{claim min distance is lambda} Let $\rho >0$ be the constant chosen in Prop. \ref{prop existence of Morse functions}.  For any $x\in(-\varepsilon,\varepsilon)\times S(V)$,  we have \[  d^{\nu_i}(x, \psi_{1}^{\nu_i}(x))= \rho,\] 
    where $d^{\nu_i}$ is the distance induced by the metric $g^{\nu_i}$.
    \end{lemma}
    \begin{proof}
        For simplicity let us assume that $\sigma=1$: the proof in the negative case follows by replacing $x$ with $\big(\psi_{1}^{\nu_i}\big)^{-1}$.  On $(-\varepsilon,\varepsilon)\times S(V)$, Lemma \ref{lem Hamiltonian flow of morse function} shows that the Hamiltonian vector field $X_{\nu_i}$ of $h^{\nu_i}$ is the vector field $\rho\mathbf{R}$,  i.e.  a (positive) multiple of the Reeb vector field.  According to Rem.  \ref{rem induced metric on neck},  $M$ comes equipped with the metric $g^{\nu_i}$.  Since this metric is the one induced by $J^{\nu_i}$ on $(-\varepsilon,\varepsilon)\times S(V)$,  the embedding 
        \[\Phi_{\nu_i} : I_{\nu_i} \times S(V) \to M\] 
        is an isometry. \\ In $I_{\nu_i} \times S(V)$ the distance between a point $x$ and $\psi_{1}^{\nu_i}(x)$ must be at least $\rho$,  thanks to the fact that the integral lines of the Reeb vector field are minimising geodesics and the two points are clearly connected by one of those whose length is $\rho$.
    \end{proof}
    Since the constant $\rho>0$ can be arbitrarily small,  let us assume it is smaller than the minimum among the geodesic radial length of $\left(2/3\varepsilon, \varepsilon \right)\times S(V)$.  Since we are not changing the almost complex structure there,  the metric is preserved and such length is preserved for each neck-stretch parameter.\\
    \begin{lemma}\label{cor Ji subset Ii} For every $t \in [0,1]$ and for each $r\in L_i(t)$ we have 
    \[\rho \leq\int_0^1 \lvert du^i(r,s)\rvert_{g^{\nu_i}_s}^{op} ds.\]
    \end{lemma}
    \begin{proof}
        Let $r\in L_i(t)$: if $u^i(r,0)\in (-\varepsilon,\varepsilon)\times S(V)$ then the claim is an immediate consequence of Lemma \ref{claim min distance is lambda} since $\tau^{\nu_i}_pu^i(r,1)=u^i(r,0)$.  If $u^i(r,0)\notin (-\varepsilon,\varepsilon)\times S(V)$,  since there must be $t$ such that $u^i(r,t) \in \left[-\varepsilon/2, \varepsilon/2 \right]\times S(V)$ by definition of $L_i(t)$.  Let $[t_0,t_1]$ be a connected component contained in 
        \[u^i(r,-)^{-1}\left(\left(\left[-\varepsilon, -\frac{2}{3}\varepsilon\right]\cup \left[\frac{2}{3}\varepsilon, \varepsilon\right]  \right)\times S(V)\right). \] Then
        \[ \rho \leq \int_{t_0}^{t_1} \lvert du^i(r,s)\rvert_{g^{\nu_i}}^{op} ds \leq  \int_0^1 \lvert du^i(r,s)\rvert_{g^{\nu_i}_s}^{op} ds\]
        since the geodesic radial length of $\left(2/3\varepsilon, \varepsilon \right)\times S(V)$ is greater than $\rho$.  We remind the reader that we dropped the $s$-dependency in the first integral thanks to Rem. \ref{rem induced metric on neck}.
    \end{proof}
    Since by assumption our family of solutions passes through $\{0\}\times S(V)$, from now on we will only consider those $t\in [0,1]$ such that the path \[ u^i(-,t) : \R \to M\] intersects $ \{0\}\times S(V)$. 
    \begin{lemma}\label{claim Ji non empty} Let $\mu$ be the Lebesgue measure on $\R$, let $R_i$ be the geodesic (radial) length of the neck $[-\varepsilon/2,\varepsilon/2]\times S(V)$. For any $i$ and $t$ such that $L_i(t)$ is non-empty, we have
    \[ \ \mu(L_i(t)) \geq \frac{R_i}{C}.\]
    \end{lemma}
    \begin{proof}
        We already observed that $L_i(t)$ is a Lebesgue-measurable, non-empty subset of $\R$. To prove the lower-bound on its measure, let us argue by contradiction. The smooth path \[ u^i(-,t) : L_i(t) \to M\] either intersects both connected components of \[\partial \left([-\varepsilon/2,\varepsilon/2]\times S(V)\right)=\{\pm\varepsilon/2\}\times S(V),\] or it starts at one of them, reaches $\{0\}\times S(V)$ and then ends at the same connected component.  In both cases,  its length will be bounded above by $C\mu(L_i(t)) < R_i$.  This contradicts the fact that the geodesic distance of $\{\pm\varepsilon/2\}\times S(V)$ from $\{0\}\times S(V)$ is $R_i/2$.
    \end{proof}
    We are now ready to conclude the proof of Theorem \ref{thm lower bound on energy for strips} by showing that Case $2$ cannot happen in our setting. Lemma \ref{cor Ji subset Ii} together with the Cauchy-Schwarz inequality implies for every $r \in L_i(t)$
    \begin{align*}
        \rho &\leq \int_0^1 \lvert du^i(r,s)\rvert_{g^{\nu_i}_s}^{op} ds \leq \sqrt{\int_0^1 \left(\lvert du^i(r,s)\rvert_{g^{\nu_i}_s}^{op}\right)^2ds}.
    \end{align*} 
    Let $t_i\in [0,1]$ be any value such that the path $u^i(-,t_i)$ intersects $\{0\}\times S(V)$.  Using Fubini-Tonelli we have:
    \begin{align*}
        E(u^i)&= \int_{\R \times [0,1]} \left(\lvert du^i(r,s)\rvert_{g^{\nu_i}_s}^{op}\right)^2dsdr \\
        &\geq \int_{L_i(t_i) \times [0,1]} \left(\lvert du^i(r,s)\rvert_{g^{\nu_i}_s}^{op}\right)^2dsdr\\ &=  \int_{L_i(t_i)}\left(\int_0^1\left(\lvert du^i(r,s)\rvert_{g^{\nu_i}_s}^{op}\right)^2ds\right)dr\\ 
        &\geq \frac{R_i}{C}\rho^2
    \end{align*}
    where in the last step we used Lemma \ref{claim Ji non empty}. Notice that the lower-bound does not depend on the chosen $u^i \in U_i$, and it clearly tends to infinity for increasing neck-stretch parameter $i$ since 
    \[ \lim_{i\to +\infty} R_i = +\infty.\]
    This is a contradiction with the assumption that the energy of the family $\{u_i\}_i$ was uniformly bounded above by a constant $\Upsilon$.
\end{enumerate}
\end{proof}

\section{The energy filtration and modified Floer cohomology}
In this section we adapt \cite[Section 3]{Ono95} to our setting. \\ \noindent Since $[\omega]$ is a rational cohomology class,  by multiplying it with a suitable scalar we can assume that $[\omega] \in H^2(M;\Z)$.  This implies that the action functional has a discrete action spectrum.  Since every discrete subset of the reals must be countable,  for each Hamiltonian $H$,  the action functional $\mathcal{A}_H$ will have at most countably many critical values.\\
\begin{lemma}\label{lem symplectic form rational is rational in mapping torus} Let $(M,\omega)$ be a symplectic manifold with integral symplectic form,  and let $\tau$ denote the composition of Dehn twists.  Then the class \[ [\omega_{\tau}] \in H^2(M_{\tau}; \R) \] induced by $\omega$ lifts to integer valued cohomology
\end{lemma}
\begin{proof}
In order to apply the Mayer-Vietoris l.e.s. for $M_{\tau}$,  let us find two suitable open subsets $U$ and $W$ that cover the mapping torus.  Up to some obvious isotopies in order to ensure they are actually open subsets and have a nice enough intersection, we set
\begin{align*}
U &:= \left( M \setminus \bigcup_{l=1}^m V_l \right) \times S^1 \\ W &:= \bigcup_{l=1}^m\left( V_l\times [0,1] \right) / (x,0)\sim (\tau_{V_l}(x),1)
\end{align*}
and let us denote with $\iota_U$, (resp. $\iota_W$) the obvious inclusion $\iota_U : U \to M_{\tau}$ (resp. $\iota_W : W \to M_{\tau}$).  Notice that $U\cap W$ is homotopy equivalent to $\bigcup_{l=1}^m S(V_l)\times S^1$.  Since the Lagrangian spheres $V_l$'s are assumed to be disjoint,  we know that 
\[\tau(z)=\tau_{V_l}(z) \ \Leftrightarrow z \in V_l.\] 
In particular,  $W$ is homotopy equivalent to the mapping torus for the map $\tau : \bigcup_lV_l \to \bigcup_lV_l$
\begin{center}
\begin{tikzcd}[column sep = 1.2cm]
 H_2(U;\R)\oplus H_2(W;\R)\arrow{r}{{\iota_U}_*-{\iota_W}_*}   & H_2(M_{\tau};\R)  \arrow{r}{\partial} & H_1(U \cap W;\R)  \arrow{r}{j_U\oplus j_W} & H_1(U;\R)\oplus H_1(W;\R)
\end{tikzcd}
\end{center}
where $j_U : U\cap W \to U$ is the canonical inclusion, similarly $j_W$.\\
\noindent We first claim that $H_2(W;\R)=0$.  $W$ is the disjoint union of the mapping tori of each Dehn twist.  If $\dim M \geq 6$ this is a consequence of the fact that each Lagrangian sphere is at least $2$-connected.  In the case $\dim M = 4$,  the second homology group of the mapping torus of the antipodal map on $S^2$ is torsion,  hence its real-coefficient homology vanishes.\\
\noindent Notice that the map induced by the inclusion $H_1(U\cap W;\R) \to H_1(U;\R)\oplus H_1(W;\R)$ is injective: a quick application of real-coefficients Gysin sequence together with K\"unneth Theorem shows that \[H_1(U \cap W;\R)\cong \R\langle [S^1]\rangle\] Establishing injectivity at this point is easy,  hence the map $H_2(U;\R)\to H_2(M_{\tau};\R)$ has to be surjective, and as a consequence we can write any $\alpha \in H_2(M_{\tau})$ as ${\iota_U}_*\alpha'$ for $\alpha' \in H_2(U)$. \\
\noindent K\"unneth theorem applied to $H_2(U)$ reveals that 
\[H_2(U) \cong H_2(M \setminus \bigcup V_l ; \R )\oplus \left(H_1(M \setminus \bigcup V_l; \R)\otimes H_1( S^1 ; \R)\right).\]
By construction of $\omega_{\tau}$ this class vanishes on classes arising from the base of the mapping torus,  hence the only non-trivial contributions come from 
\[ H_2\left(M \setminus \bigcup V_l ; \R\right),\] 
where we know it pairs integrally with $\Z$-coefficients cycles.  
\end{proof}

\begin{lemma}\label{lemma action functional coincides between neck parameters} For each neck-stretch parameter $\nu$,  the set of critical values of $\mathcal{A}_{h^{\nu}}$ coincides with the one of $\mathcal{A}_{h^{\nu_0}}$
\end{lemma}
\begin{proof}
We first observe that any periodic orbit of the Hamiltonian vector field $X_{h^{\nu}}$ is entirely contained in $M'$. This is because the fixed points of $\tau_p^{\nu}$ are entirely contained in $M'$ and integral lines of $X_{h^{\nu}}$ are contained in the level sets of $h^{\nu}$.  Since $h^{\nu}\equiv h^{\nu_0}$ on $M'$ by construction,  we conclude.
\end{proof}
Let $\{\nu_i\}_i$ be a sequence of neck-stretch parameters converging to $+\infty$ for $i \to +\infty$ and let $\{r_k\}_k$ be any increasing sequence of real numbers such that it satisfies the following conditions:
\begin{itemize}
\item $\lim\limits_{k\to +\infty} r_k = +\infty$\\
\item $\{r_k\}_k$ does not contain critical values of $\mathcal{A}_{h^{\nu_0}}.$
\end{itemize}
The second condition is achieved by observing we are avoiding countably many critical values.  Lemma \ref{lemma action functional coincides between neck parameters} shows that in this way we are avoiding all the critical points of all the action functionals associated with our neck-stretching.
\begin{defi} Let $F^k{C_i}^{\bullet}$ be the subcomplex of ${C_i}^{\bullet}:={C_i}^{\bullet}\left(\tau_p^{\nu_i}; \Lambda_{\omega}\right)$ defined as follows \[  F^k{C_i}^{\bullet} := \left\lbrace \sum \xi_{[\gamma,u]}\cdot [\gamma,u] \in {C_i}^{\bullet} \mid \xi_{[\gamma,u]}=0 \text{ if } \mathcal{A}_{h^{\nu_i}}([\gamma,u])< r_k\right\rbrace.\]
\end{defi}
\noindent Notice that, for $k\leq j$, we have $F^j{C_i}^{\bullet}\subseteq F^k{C_i}^{\bullet}$. This motivates the following lemma.
\begin{lemma} The filtration defined above is complete and exhaustive.  In other words we have the following isomorphisms of $\Lambda_{\omega}$-vector spaces \begin{align*}
F^k{C_i}^{\bullet} &\cong \varprojlim_{j\to +\infty}F^k{C_i}^{\bullet}/F^j{C_i}^{\bullet} \\
{C_i}^{\bullet}&\cong \varinjlim_{k\to -\infty} F^k{C_i}^{\bullet}.
\end{align*}
\end{lemma}
\begin{proof}
The proof is a consequence of the fact that the same statements hold for the Novikov field: $\Lambda_{\omega}$ is the completion of the group ring of a suitable quotient of $\pi_1(\Omega_{\tau})$ with respect to the additive valuation given by the action functional. 
\end{proof}
\begin{defi}
For $k\leq j$, we define the so called \textit{relative} Floer cohomology of the pair $\left(F^k{C_i}^{\bullet},F^j{C_i}^{\bullet}\right)$ as follows: \begin{equation*}\label{def rel cohomology}
HF^{\ast}_{(k,j)}\left(\tau_p^{\nu_i}\right) := H^{\ast}\left(F^k{C_i}^{\bullet}/F^j{C_i}^{\bullet} \right).
\end{equation*}
\end{defi} 
\begin{defi}\label{def modified floer groups} The \textit{modified} fixed point Floer cohomology groups of $\tau_p^{\nu_i}$ are defined as follows:
\[ \widehat{HF^{\ast}}\left(\tau_p^{\nu_i}; \Lambda_{\omega}\right) := \varinjlim_{k\to -\infty} \varprojlim_{j\to +\infty}  HF^{\ast}_{(k,j)}\left(\tau_p^{\nu_i}\right).\]
\end{defi}

\begin{rem} The modified groups $ \widehat{HF^{\ast}}\left(\tau_p^{\nu_i}; \Lambda_{\omega}\right)$ have a natural $\Lambda_{\omega}$-module structure and do not depend on the choice of $\{r_k\}$  (See \cite[Lemma 3.2-3.3]{Ono95})
\end{rem} 

Since we assume that $[\omega]\in H^2(M;\Q)$, the spectrum of the action functional is discrete.  We can then additionally assume that our sequence of increasing real numbers $\{r_k\}_k$ satisfies the following additional condition:
\begin{itemize}
\item there exists $\theta > 0$ such that in the $\theta$-neighborhood of each $r_k$ there are no critical values of any $A_{h^\nu}$
\end{itemize}
By choosing the sequence of neck-stretch parameters $\{\nu_i\}_i$ appropriately,  we can assume that \begin{equation*}
\| h^{\nu_{i+1}}-h^{\nu_{i}}\|_{C^1}< \frac{\theta}{2}.
\end{equation*}
Sard - Smale then ensure we can find a regular path $h_s^{\nu_i}$ connecting $h^{\nu_{i}}$ to $h^{\nu_{i+1}}$ such that 
\begin{equation*}
\int_{-\infty}^{+\infty} \max_{x \in M} \left\lvert \frac{\partial h_s^{\nu_i}}{\partial s} \right\rvert ds < \theta
\end{equation*}
which we can use to define our continuation maps 
\[\chi_{i,i+1} \colon HF^{\ast}\left(\tau_p^{\nu_{i}}\right)\to HF^{\ast}\left(\tau_p^{\nu_{i+1}}\right). \] 
\noindent In our setting,  these continuation maps preserve the filtration:
\begin{theorem}[{\cite[Theorem 3.4]{Ono95}}]\label{thm 3.4 Ono continuation maps preserves filtration} Assume that $[\omega]$ can be represented by an integral class.  The family of continuation maps $\{\chi_{i,i+1}\}_i$ induces a family of (invertible) continuation maps between relative Floer cohomology groups with respect to the filtration $\{ r_k\}_k$ above.
\[\chi_{i,i+1}^{(k,j)} \colon HF^{\ast}_{(k,j)}\left(\tau_p^{\nu_{i}}\right)\to HF^{\ast}_{(k,j)}\left(\tau_p^{\nu_{i+1}}\right)\]
for any pair $(k,j)$ such that $k <j$.  These maps are compatible with the homomorphisms in the colimit and inverse limit defining the modified Floer groups.
\end{theorem}

\section{Proof of Theorem \ref{theorem 1 case 2}}
\noindent Let \[CM^{\bullet} :=\bigoplus_{\ast\in \N} CM^{\ast}\left(h';\Z_2\right)\otimes \Lambda_{\omega}\] be the Morse cochain complex with coefficients in the Novikov field $\Lambda_{\omega}$ associated to the Morse function $h' : M'\to \R$.  We will denote with $HM^{\ast}\left(h';\Lambda_{\omega} \right)$ the cohomology of $\left(CM^{\bullet},\partial_{\text{Morse}}\otimes 1_{\Lambda_{\omega}}\right)$.  By the universal coefficient theorem, we immediately have:
\[HM^{\ast}\left(h';\Lambda_{\omega} \right)\cong HM^{\ast}\left(h';\Z_2 \right)\otimes \Lambda_{\omega}.\]
Without loss of generality we can assume that the increasing sequence of real numbers $\{r_k\}_k$ defined previously also misses the critical values of any $h'$.  We can now define a filtration for our Morse complex as follows:
\begin{defi} 
Let $F^k{CM}^{\bullet}$ be the subcomplex of ${CM}^{\bullet}$ defined as
\[  F^k{CM}^{\bullet} := {CM}^{\bullet}\otimes \Lambda^k_{\omega},\] 
where \[ \Lambda^k_{\omega} := \left\lbrace \sum_{A \in N}\xi_A\cdot A \in \Lambda_{\omega} \mid \xi_A=0 \text{ if } \text{ev}_{\omega}(A) < r_k \right\rbrace.\]
\end{defi}
\noindent For $k\leq j$, we can define a relative version of Morse cohomology:
\begin{align*}
HM_{(k,j)}^{\ast}\big(h'\big) &:= H^{\ast}\left(F^k{CM}^{\bullet}/F^j{CM}^{\bullet} \right)\\
&\cong H^{\ast}\left({CM}^{\bullet}\otimes \Lambda^k_{\omega}/\Lambda^j_{\omega} \right)\\
&\cong H^{\ast}\left({CM}^{\bullet}\right)\otimes \Lambda^k_{\omega}/\Lambda^j_{\omega}.
\end{align*}
A simple algebraic argument gives the following result
\begin{lemma}\label{lem modified morse homology is morse homology} There is an isomorphism of $\Lambda_{\omega}$-modules \[HM^{\ast}\left(h'; \Lambda_\omega\right) \cong \varinjlim_{k \to -\infty} \varprojlim_{j\to +\infty} HM^{\ast}_{(k,j)}\left(h'\right). \]
\end{lemma}

\begin{prop}\label{prop existence filtration with gap less than lambda_i} Let $\{r_k\}_k$ be the filtration defined in the previous section. Let $\{\lambda_i\}_i$ be the increasing sequence of numbers from Theorem \ref{thm lower bound on energy for strips}. For each $i$ and $k$, let $a_{k,i}\in \N$ be the biggest index such that  
 \[ | r_k - r_{a_{k,i}}| < \lambda_i,\]
then:
\begin{enumerate}
\item $\lim\limits_{i \to \infty} r_{a_{k,i}} = +\infty$
\item For each index $i$,  there is a natural isomorphism of relatively graded $\Lambda_{\omega}$-vector spaces  \[ HF^{\ast}_{(k,a_{k,i})}\left(\tau_p^{\nu_i}\right) \cong HM^{\ast}_{(k,a_{k,i})}\big(h'\big).\]
\end{enumerate}
\end{prop}
\begin{proof}
The first condition is a consequence of Theorem \ref{thm lower bound on energy for strips} which shows that $\lambda_i \to +\infty$.  Let $[x_-,v]$ be a representative of a non-zero element in $F^{k}{C_i}^{\bullet}/F^{a_{k,i}}{C_i}^{\bullet}$.  Let us assume there is a $J^{\nu_i}$-holomorphic curve $u^i$ realizing $\partial \left([x_-,v]\right)=[x_+,v \# u^i]$ and passing through any of the Lagrangian sphere $V_l$.  By Theorem  \ref{thm lower bound on energy for strips}
 \[ E(u^i) = \mathcal{A}_{h^{\nu_i}}([x_+, v\  \# \ u^i])-\mathcal{A}_{h^{\nu_i}}([x_-, v])\geq \lambda_i,\] 
 we see that $[x_+, v \ \# \ u]\in F^{a_{k,i}}{C_i}^{\bullet}$ and hence the differential defined by $u^i$ is the zero map in $F^{k}{C_i}^{\bullet}/F^{a_{k,i}}{C_i}^{\bullet}$. Put differently, the pseudo-holomorphic curves realising non-trivial differentials in $F^{k}{C_i}^{\bullet}/F^{a_{k,i}}{C_i}^{\bullet}$ cannot pass through any Lagrangian sphere $V_l$.  By the same reasoning in \cite[Theorem 4.7]{Ono95},  the \textit{low energy} PSS isomorphism (\cite[Theorem 6.1]{HS95}) immediately gives us an isomorphism of relatively graded vector spaces  
 \[ HF^{\ast}_{(k,a_{k,i})}\left(\tau_p^{\nu_i}\right) \cong HM_{(k,a_{k,i})}^{\ast}\left(h'\right).\] 
 Naturality with respect to continuation maps is a consequence of the fact that the general PSS isomorphism intertwines continuation map in Floer theory and Morse theory (in the sense of \cite{Sch93}).  It is explicitly proved in \cite[Theorem 2]{Dju14}. 
\end{proof}

\begin{rem} In \cite{Dju14} the naturality of the PSS isomorphism is proved by analysing the boundary of the relevant moduli spaces. Therefore the same reasoning applies in our setting.  Moreover,  as observed in \cite[Chapter 12.1]{McDS12},  when working with small enough autonomous Hamiltonians then the PSS maps coincide with the classical isomorphism due to Floer and explained in \cite{HS95}. For this reason, the term \textit{low energy} PSS isomorphism seemed quite evocative to us.
\end{rem}

\begin{proof}[Proof of Theorem \ref{theorem 1 case 2}]
For an index $k$,
\begin{align}
\varprojlim_{j\to +\infty}HF^{\ast}_{(k,j)}\left(\tau_p^{\nu_0}\right) &\cong \varprojlim_{j\to +\infty}\varprojlim_{i\to +\infty}  HF^{\ast}_{(k,j)}\left(\tau_p^{\nu_i}\right) \label{line 1 final proof} \\
&\cong \varprojlim_{i\to +\infty} HF^{\ast}_{(k,a_{k,i})}\left(\tau_p^{\nu_i}\right) \label{line 2 final proof}\\
&\cong \varprojlim_{i\to +\infty} HM^{\ast}_{(k,a_{k,i})}\left(h'\right) \label{line 3 final proof} \\
&\cong \varprojlim_{j\to +\infty}\varprojlim_{i\to +\infty} HM^{\ast}_{(k,j)}\left(h'\right) \\
&\cong \varprojlim_{j\to +\infty} HM^{\ast}_{(k,j)}\left(h'\right) \label{line 6 final proof}
\end{align}
Line \ref{line 1 final proof} is justified by Theorem \ref{thm 3.4 Ono continuation maps preserves filtration} which shows that every continuation map is an isomorphism.  Line \ref{line 2 final proof} can be proved as follows: a limit of a bigraded inverse system is canonically isomorphic to the limit along the diagonal.  More abstractly,  for any small category $\mathcal{C}$, the diagonal functor $\Delta : \mathcal{C} \to \mathcal{C}\times \mathcal{C}$ is initial, hence pulling back along it preserves limits. Now it is immediate to see we can arrange the diagonal to coincide with the indexes $a_{k,i}$'s.  Line \ref{line 3 final proof} is a direct application of Prop. \ref{prop existence filtration with gap less than lambda_i}. Since the Morse function $h'$ is independent on $i$, the Morse continuation maps are simply the identity maps, in particular the inverse limit is taken over isomorphisms, proving line \ref{line 6 final proof}.  
\noindent We have therefore showed that  
\begin{equation}\label{eq isomorphism floer and morse novikov ring}
\varprojlim_{j\to +\infty}HF^{\ast}_{(k,j)}\left(\tau_p^{\nu_0}\right) \cong \varprojlim_{j\to +\infty} HM^{\ast}_{(k,j)}\left(h'\right).
\end{equation}
We now take the colimit for $k \to -\infty$ in Eq. \ref{eq isomorphism floer and morse novikov ring} and obtain the following isomorphism:
\begin{equation*}
\widehat{HF^{\ast}}\left(\tau_p^{\nu_0}; \Lambda_{\omega}\right) \cong \varinjlim_{k \to -\infty} \varprojlim_{j\to +\infty} HM^{\ast}_{(k,j)}\left(h'\right) \cong HM^{\ast}\left(h'; \Lambda_\omega\right)
\end{equation*}
By Prop. \ref{prop appendix isomorphism modified and standard floer cohomology} the leftmost term is canonically isomorphic to the \textit{standard} Floer cohomology group $HF^{\ast}\left(\tau_p^{\nu_0}; \Lambda_{\omega}\right)$. The rightmost term, being the usual Morse cohomology for the Morse function $h'$, is canonically isomorphic to $H^{\ast}\left(M,V;\Lambda_{\omega}\right)$.\\ The case with more than one Dehn twists follows analogously, with the only difference that the Morse function $h'$ under consideration will have its gradient either \textit{inward} or \textit{outward} along the connected components of the boundary of $M'$ depending on the sign of $\sigma_l$. The resulting Morse cohomology groups, thanks to our conventions, will then be isomorphic to $H^{\ast}\left(M \setminus C_-, C_+; \Lambda_{\omega}\right)$.
\end{proof}
\begin{rem} The assumption on $[\omega] \in H_2(M;\Q)$ for a symplectic manifold of dimension at least $6$ can be dropped as follows.  Let $\eta$ be a closed $2$-form realizing a perturbation such that \[ [\omega + \eta] \in H^2_{dR}(M; \Q). \]
The form $\eta$ can be taken small enough in order to ensure $\omega + \eta$ is still non-degenerate.  We want to show now that we can tweak $\eta$ inside its cohomology class to make it vanish on each Lagrangian sphere $V_l$.  Since $\dim M \geq 6$,  the Lagrangian spheres are $2$-connected, hence for each index $l$,  we can find $1$-forms $\alpha_l \in \Omega^1_{dR}(V_l; \R)$ such that \[ \left.\eta\right|_{V_l} = d\alpha_l.\] 
Using smooth bump functions $\phi_l$ supported around each Lagrangian sphere,  we can define the $2$-form
\[ \eta' := \eta -\sum_{l} d(\phi_l\alpha_l).\]
Clearly we have $[\eta]=[\eta']$ and by construction $\eta'$ vanishes on each Lagrangian sphere. This implies that each $V_l$ remains Lagrangian under this perturbation.  Since Floer cohomology is invariant under this procedure we can run the proof of the theorem for $(M,\omega + \eta')$.
\end{rem}

\section{Proof of Theorems \ref{theorem vanishing c} and \ref{theorem vanishing c for surfaces}}
This section is entirely devoted to prove Theorem \ref{theorem vanishing c},  and its adaption to the surface case.
For simplicity we will prove the case $\sigma=1$.  The general case is treated in the same way.\\
\noindent We start by making precise the assumptions on the Lagrangian spheres.  Let $w: (D^2,\partial D^2) \to (M, V)$ be a smooth map.  We can find a symplectic trivialization,  unique up to homotopy,  of 
\[ \psi : w^*TM \cong D^2\times \C^n.\] 
Using $\psi$ we can define a map from the boundary of the disk into the Lagrangian Grassmannian where there is a preferred class, called the \textit{Maslov class} 
\[ \mu \in H^1(\Lambda(\C^n);\Z).\] 
The Maslov index homomorphism is then given by the following evaluation: 
\[ I_{\mu,  V_l}(w) := \mu(\partial D^2).\]
In our setting, i.e.  with the (small) autonomous Hamiltonian $h^{\nu}$ and symplectomorphism $\phi=\tau_V$,  a pseudo-holomorphic half-strip is a function \[u : \R^+ \times \R \to M\] which is a solution of 
\begin{equation}\label{eq floer symplectomorphism half strips} 
\begin{cases} &\partial_r u + J_t^{\nu}\partial_t u - \nabla_{\nu} h^{\nu}=0 \\ 
&\tau_V\circ u(r,t+1)= u(r,t)\\
 &u(0, t) \in V \ \forall t \\
  &\lim\limits_{r\to + \infty} u(r,t)=x_{+}(t) \in \text{Crit}(\mathcal{A}_{h^{\nu},\gamma_0}).\end{cases}
 \end{equation}
As we did for the full strip case,  we can always consider solutions of the equivalent system:
\begin{equation}\label{eq floer perturbed symplectomorphism half strips}
 \begin{cases} &\partial_r v + {J'}^{\nu}_t\partial_t v =0 \\ &\tau_p^{\nu}\circ v(r,t+1)= v(r,t)\\ &v(0, t) \in V \ \forall t\\ &\lim\limits_{r\to + \infty}  v(r,t)=x_{+}(0) \in \text{Fix}(\tau_p^{\nu})\end{cases} 
 \end{equation}
via the identification \[ v(s,t):= (\psi_t^{{\nu}})^{-1}u(s,t)\]
Since the path $x_{+}(t)$ is completely determined by the dynamics of the vector field $X_{h^{\nu}}$ and the starting point $x_+(0)$, we will suppress the dependency on $t$  from the notation of $x_+(t)$.\\
Let us denote with \[\mathcal{M}^{\nu}_k(V,x_+)\] the space of solution of Eq. \ref{eq floer symplectomorphism half strips}, or equivalently Eq. \ref{eq floer perturbed symplectomorphism half strips}, such that $\mu_{CZ}\left(x_+\right)=k$.
\begin{lemma}\label{lem Lagrangian sphere is monotone} In the same setting as Theorems \ref{theorem vanishing c} or Theorem \ref{theorem vanishing c for surfaces},  the Lagrangian sphere $V$ is monotone,  i.e.  the Maslov index of a pseudo-holomorphic disk $w$ with boundary on $V$ is proportional to its area: 
\[ I_{\mu,  V}(w) = \lambda \int_{D^2}w^*\omega \ \ \text{for some } \lambda > 0.\]
\end{lemma}
\begin{proof}
Let us start with the case $\Sigma$ is a Riemann surface of genus at least $2$.  Remember that $V$ is assumed to be essential, i.e. the map induced by inclusion \[i_* : \pi_1(V,\{pt.\}) \to \pi_1(\Sigma,\{pt.\}) \] is non-trivial for each $\{pt.\} \in V$. This immediately implies the absence of non-constant disks whose boundary lie in $V$ since $\pi_2(\Sigma,V)=0$. This proves the lemma in the surface case.\\ If $\dim M \geq 4$, any Lagrangian sphere is simply connected.  Therefore an application of the long exact sequence in homotopy for the pair $(M,V)$ gives us a surjection 
\[j_* : \pi_2(M) \twoheadrightarrow \pi_2(M,V)\] 
induced by the canonical map $j : (M,\emptyset) \to (M,V)$.  This means that each disk, up to homotopy, admits a spherical representative in $M$. Let $a \in \pi_2(M)$ be a representative for any such disk. Since
\[ \int_{S^2} a^*\omega = \int_{D^2}(j_*a)^*\omega \] and 
\[ I_{\mu,  V}(a)=2c_1(j_*a),\] 
(see \cite[Rem 2.3.ii]{Oh93}) the monotonicity on $M$ implies the monotonicity of $V$
\end{proof}
By \cite[Lemma 3]{Sei96},  in the surface with genus at least $2$ case,  the Novikov field $\Lambda_{\omega}$ is isomorphic to $\Z_2$. \\ 
The next lemma explains how any half-cylinder defines an element in the Novikov cover of $\Omega_{\tau_V}$.  Since we have to choose a reference path  $\gamma_0 \in \Omega_{\tau_V}$,  let us assume that $\gamma_0$ is entirely contained in $V$.  
\begin{lemma}\label{lem def c in the case dim M is 4} A solution $u : \R^+ \times \R \to M$ of Eq.  \ref{eq floer symplectomorphism half strips} defines a preferred lift $\widetilde{x_+}$ of $x_+$ in the Novikov cover $\widetilde{\Omega_{\tau_V}}$
\end{lemma}
\begin{proof}
In the surface of genus at least $2$ case,  this is true since the Novikov cover is the trivial cover of $\Omega_{\tau_V}$.  Let us assume $\dim M \geq 4$.  Since $V$ is simply-connected,  it is easy to see that we can find an homotopy \[W: [0,1]\times [0,1] \to V\] from $\gamma_0$ to $u(0,t)$ such that $W(s,1)=\tau_VW(s,0)$.  Using it we can define the lift \[ [(x_+,u\# W)] \in \widetilde{\Omega_{\tau_V}} \] of $x_+$.  Any two possible choice of the homotopies $W_1,W_2$ will define the same lift in the Novikov cover since they will have vanishing symplectic area due to the fact that their images lie in the Lagrangian sphere $V$.
\end{proof}
\noindent For each lift $\widetilde{x_+}$ of $x_+$ in the Novikov cover $\widetilde{\Omega_{\tau_V}}$,  let us set \[ \mathcal{M}^{\nu}_k(V,\widetilde{x_+})\subseteq \mathcal{M}^{\nu}_k(V,x_+)\] as the subset of half-strips whose preferred lift of $x_+$ coincides with $\widetilde{x_+}$.
\begin{lemma}\label{lem existence regular almost complex structure for definining c} For each neck stretch parameter $\nu_i$,  we can find generic suitable almost complex structures $J^{\nu_i}$ and regular Hamiltonians $h^{\nu_i}$ such that the following are simultaneously satisfied: 
\begin{enumerate}
\item We can define the groups  $HF^{\ast}(\tau_p^{\nu_i}; \Lambda_{\omega})$
\item $\mathcal{M}^{\nu}_k(V,\widetilde{x_+})$ is a smooth manifold of dimension $k$.
\end{enumerate}
\end{lemma}
\begin{proof}
We already proved the existence of generic almost complex structure which are standard around $V$ and regular Hamiltonians that satisfies 1.  Standard results in Floer theory show that we can find almost complex structures standard on $V$ and regular Hamiltonians such that $\mathcal{M}^{\nu}_k(V,x_+)$ (compare with \cite[Thm 4.1]{Alb05} and subsequent remark) is a smooth manifold. We observe that we can tweak our almost complex structure far from $V$ since every $J^{\nu}$-holomorphic half-cylinder has to leave the chosen tubular neighborhood of the Lagrangian sphere.  Since the two sets of generic almost complex structures are comeager, they must intersect and hence we can use an almost complex structure $J^{\nu}$ which ensures that both 1. and 2. hold simultaneously.
\end{proof}
From now on we will always work with an almost complex structure $J^{\nu_i}$,  for the appropriate neck-stretch parameter $\nu_i$, coming from Lemma \ref{lem existence regular almost complex structure for definining c}.
\begin{defi}
For each neck-stretch parameter $\nu_i$ we define the cochain $c_{\nu_i}\in CF(\tau_{\nu_i}; \Lambda_{\omega})$ as follows:
\begin{equation*}
c_{\nu_i} = \sum_{\widetilde{x_+} \in \widetilde{\Omega_{\tau}}} \left\lvert \mathcal{M}^{\nu}_0(V,\widetilde{x_+}) \right\rvert_2 \cdot \widetilde{x_+}
\end{equation*}
where $\lvert \hspace*{.5mm} \cdot \hspace*{.5mm} \rvert_2$ denotes the cardinality of the set modulo $2$.
\end{defi}

\begin{lemma}\label{lem class c is a cycle} For each neck-stretch parameter $\nu_i$,  the class $c_{\nu_i}$ defined above is well-defined and it is a cycle,  i.e.  we have 
\[ \partial c_{\nu_i} = 0. \]
\end{lemma}
\begin{proof}
Well-definedness in this context means that the $0$-dimensional manifold $\mathcal{M}^{\nu}_0(V,\widetilde{x_+})$ is compact for each $\widetilde{x_+} \in \widetilde{\Omega_{\tau}}$.  This is a consequence of Lemma \ref{lem Lagrangian sphere is monotone}.  By looking at the compactification of $\mathcal{M}^{\nu}_1(V,\widetilde{x_+})$,  the lack of disk or sphere bubbling (thanks to transversality and Lemma \ref{lem Lagrangian sphere is monotone}) implies that the boundary of the $1$-dimensional manifold is comprised of index $0$ half-cylinder together with an index $1$ $J^{\nu_i}$-holomorphic strip.  These elements represent exactly $\partial c_{\nu_i}$,  therefore we can conclude that $c_{\nu_i}$ is a cycle.
\end{proof}
\begin{lemma}\label{lem c is natural} Let 
\[\chi_{i,i+1} : CF^{\ast}\left(\tau_p^{\nu_i}; \Lambda_{\omega}\right) \to CF^{\ast}\left(\tau_p^{\nu_{i+1}}; \Lambda_{\omega}\right)\] 
be a continuation map associated with a regular path connecting $(J^{\nu_i},h^{\nu_i})$ to $(J^{\nu_{i+1}},h^{\nu_{i+1}})$. Then 
\[ {\chi_{i,i+1}}_*[c_{\nu_i}] = [c_{\nu_{i+1}}]\in HF^{\ast}\left(\tau_p^{\nu_{i+1}}; \Lambda_{\omega}\right).\]
\end{lemma}
\begin{proof}
At the cochain level,  the cycle $\chi_{i,i+1}c_{\nu_i}$ is given by an appropriate count of the following con\-fi\-gu\-ra\-tions: 
\begin{center}
\includegraphics[scale=1]{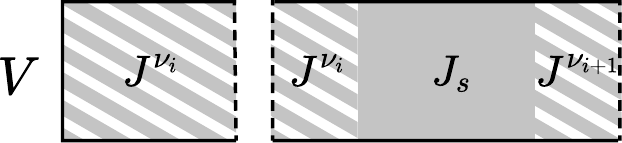} 
\end{center}
where the strips are the index $0$ strips defining the continuation map $\chi_{i,i+1}$. By the gluing theorem we can find gluing parameters such that the count of these configurations is equal to counting $\overline{J}$-holomorphic half-strips where $\overline{J}$ is the (generic) almost complex structure represented in the following picture:
\begin{center}
\includegraphics[scale=1]{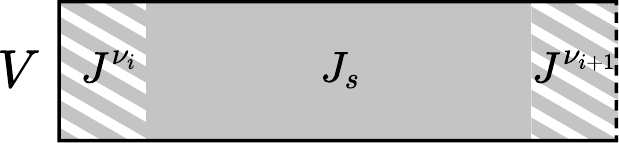} 
\end{center}
Similarly we obtain the regular Hamiltonian $\overline{h}$. Now in order to prove that ${\chi_{i,i+1}}_*[c_{\nu_i}]=[c_{\nu_{i+1}}]$ we fix a regular path 
\[\overline{J}_s : \left(\overline{J},  \overline{h} \right) \rightsquigarrow \left(J^{\nu_{i+1}},h^{\nu_{i+1}} \right).\]
For each $\widetilde{x_+} \in \widetilde{\Omega_{\tau}}$ consider the parametric moduli space:
\begin{center}
\begin{tikzcd} \mathcal{M}_1^{\text{par}}(V,\overline{J}_s,\widetilde{x_+}) \arrow{d}{\pi} & \\
\left[0,1\right] & 
\end{tikzcd}
\end{center}
by Lemma \ref{lem Lagrangian sphere is monotone}, there are no disk or sphere bubbling in this dimension. After taking the Gromov-Floer compactification,  its boundary consists of the following three pieces: 
\[\partial {\mathcal{M}_1^{\text{par}}(V,\overline{J}_s,\widetilde{x_+})}= - \mathcal{M}_0(V,\overline{J},\widetilde{x_+}) \cup \mathcal{M}_0(V,J^{\nu_{i+1}},\widetilde{x_+}) \cup \left( \text{ strip breaking at some values of the parameter } s\right).\] 
Notice that the first boundary component defines $-\chi_{i,i+1}c_{\nu_i}$, the second $c_{\nu_{i+1}}$ and the strip breaking instead amounts to $\partial(\beta)$ for some $\beta \in CF^*(\tau_p^{\nu_{i+1}},\Lambda_{\omega})$. Therefore when passing to cohomology the two cochains define the same class proving our claim.
\end{proof}
We now need to build suitable continuation maps which induces continuation maps for relative Floer cohomology, \textit{without} appealing to Theorem \ref{thm 3.4 Ono continuation maps preserves filtration} since we do not want to assume rationality of the symplectic class $[\omega]$. Following the discussion in \cite[Page 10]{Wen16}, this is still possible by working with an increasing family of Hamiltonians $\{\bar{h}^{\nu}\}_{\nu}$. The following Corollary of Prop. \ref{prop existence of Morse functions} ensures we can build such family.

\begin{cor}\label{cor we can make Hamiltonians bigger} Let $\{h^{\nu}\}_{\nu}$ be a family of Hamiltonians build in Prop. \ref{prop existence of Morse functions}, let $\{h^{\nu_i}\}_i$ be any countable subfamily of it. Since $M$ is compact,  there is a sequence of real numbers $\{d_i\}_i$ such that for each $i \in \N$ and each $x \in M$,  
\[ h^{\nu_i}(x)+d_i \leq h^{\nu_{i+1}}(x) + d_{i+1}\] 
In particular,  we can always assume that our countable family of Hamiltonians is increasing.  
The family \[ \overline{h^{\nu_i}}(x):=h^{\nu_i}(x)+d_i \]
is therefore an increasing family of Hamiltonians which satisfies property 2. of Prop. \ref{prop existence of Morse functions}, it restricts to a family of Morse functions on $M'$ and whose critical points of each $\overline{h^{\nu_i}}$ are independent of the index $i$. Up to some translation,  we can also assume that, for each index $l$,  \[ \overline{h^{\nu_i}}(V_l)=d_i.\]
\end{cor}
Thanks to the corollary, and \cite[Page 10]{Wen16}, by using the family of Hamiltonians $\{\overline{h^{\nu_i}}\}_i$, there are suitable continuation maps  \[\chi_{i,i+1} : CF^{\ast}\left(\tau_p^{\nu_i}; \Lambda_{\omega}\right) \to CF^{\ast}\left(\tau_p^{\nu_{i+1}}; \Lambda_{\omega}\right).\] 
which increase the value of the action functional of the image:
\[ \mathcal{A}_{\overline{h^{\nu_i}}}(\widetilde{x_+}) \leq \mathcal{A}_{\overline{h^{\nu_{i+1}}}}(\chi_{i,i+1}\widetilde{x_+}).\]
\noindent The last ingredient we need in order to prove Theorem \ref{theorem vanishing c} is the following adaption of Theorem \ref{thm lower bound on energy for strips} which we will not prove since the proof is \textit{mutatis mutandis} the same as the original one.
\begin{theorem}\label{thm lower bound on energy for half-cylinders} Let $\{\nu_i\}$ be a sequence of neck-stretch parameters, then there exists a sequence of positive real numbers $\{\lambda_i \}$ satisfying \[ \lim_{i\to+\infty} \lambda_i= +\infty,\] such that for every index $i$, every solution $u^i$ of Eq. \ref{eq floer perturbed symplectomorphism half strips} for the stretched manifold $(M,J^{\nu_i})$ with Hamiltonian function $\overline{h^{\nu_i}}$ has its energy bounded below by $\lambda_i$.
\end{theorem}
Let $\{d_i\}_i$ be the increasing sequence of numbers from Cor. \ref{cor we can make Hamiltonians bigger}. Without loss of generality, we can assume that the sequence of numbers $\{r_k\}_k$ from Prop. \ref{prop existence filtration with gap less than lambda_i} satisfies the additional condition:
\[ r_k < \lambda_k + d_k.\]
This is possible since the sequence $\{\lambda_k + d_k\}_k$ tends to $+\infty$. The set of critical values of $h^{\nu_i}$ is a discrete subset of $\R$ thanks to Lemma \ref{lemma action functional coincides between neck parameters}. Hence, for any $\nu_i$, the set of critical values of $\overline{h^{\nu_i}}$ is a translated copy of the one of $h^{\nu_i}$. Since a countable union of discrete subsets is countable, we can ensure that $\{r_k\}_k$ still avoids all the critical values.\\
We use this sequence $\{r_k\}_k$ to define an energy filtration on ${C_i}^{\bullet}$.
\begin{cor}\label{cor c-i lives in higher filtration steps} For each $i$,  
\[ c_{\nu_i} \in F^i{C_i}^{\bullet}.\]
\end{cor}
\begin{proof}
Let $u^i$ be a $J^{\nu_i}$-holomorphic half-strip limiting to $x_+ \in \text{Crit}\left(\bar{h}^{\nu_i}\right)$.  Following the notation of Lemma \ref{lem def c in the case dim M is 4},  $[x_+, u^i \# W] \in c_{\nu_i} $.  By definition of the action functional \begin{align*}
\mathcal{A}_{\bar{h}^{\nu_i}}\left([x_+, u^i \# W] \right) &= E(u^i) + \mathcal{A}_{\bar{h}^{\nu_i}}\left([\gamma_0, \gamma_0] \right)\\
&\geq \lambda_i + \int_{[0,1]^2}\gamma_0^*\omega + \int_0^1 \bar{h}^{\nu_i}(\gamma_0(t))dt\\
&\geq \lambda_i + d_i,
\end{align*}
where we used the facts that for the reference path we took the "constant" disk at $\gamma_0$ and that by construction $\bar{h}^{\nu_i}(V)=d_i$.  By definition of $F^i{C_i}^{\bullet}$ we conclude.  
\end{proof}
Before proceeding with the actual proof of the theorem, let us remind the reader that unless some additional assumptions are made, continuation maps do not necessarily descends to maps between \textit{relative} Floer cohomology groups and even if they do (like in the case of increasing Hamiltonians, as in Cor. \ref{cor we can make Hamiltonians bigger}) they are not necessarily invertible (as maps between \textit{relative} Floer cohomology groups).
\begin{proof}[Proof of Theorem \ref{theorem vanishing c}]
Consider the sequence of classes 
\[[c] := \left([c_{\nu_0}],[c_{\nu_1}], \dots , [c_{\nu_i}], \dots \right).\]
Lemma \ref{lem c is natural} implies
\[[c] \in \varinjlim_{k \to -\infty} \varprojlim_{i \to +\infty} HF^*_{(k,i)}\left(\tau_p^{\nu_i}\right)\]
and Cor. \ref{cor c-i lives in higher filtration steps} shows that
\begin{equation}\label{eq c=0 in inv limit}
    [c]=0 \in \varinjlim_{k \to -\infty} \varprojlim_{i \to +\infty} HF^*_{(k,i)}\left(\tau_p^{\nu_i}\right)
\end{equation}
since we showed that each cochain $c_{\nu_i}\in CF(\tau_{\nu_i}; \Lambda_{\omega})$ lives in the step of the filtration that we are quotienting out when computing the \textit{relative} Floer cohomology groups $HF^*_{(k,i)}\left(\tau_p^{\nu_i}\right)$.\\
Let $\iota$ be the canonical map coming from the universal property of the inverse limit $\varprojlim_{i}$ composed with the isomorphism provided by our usual diagonal argument:
\begin{equation*}
\iota : \varinjlim_{k \to -\infty} \varprojlim_{i \to +\infty} HF^*_{(k,i)}\left(\tau_p^{\nu_i}\right)\cong \varinjlim_{k \to -\infty} \varprojlim_{i \to +\infty} \varprojlim_{j \to +\infty} HF^*_{(k,j)}\left(\tau_p^{\nu_i}\right) \to \varprojlim_{i \to +\infty}\varinjlim_{k \to -\infty} \varprojlim_{j \to +\infty} HF^*_{(k,j)}\left(\tau_p^{\nu_i}\right).
\end{equation*} 
By Prop. \ref{prop appendix isomorphism modified and standard floer cohomology} the map $\iota$ lands in $\varprojlim_{i}HF^*\left(\tau_p^{\nu_i}\right)$ and it easy to see that 
\begin{equation*}
    \iota\left([c]\right)=\left([c_{\nu_0}],[c_{\nu_1}], \dots , [c_{\nu_i}], \dots \right) \in \varprojlim_{i \to +\infty}HF^*\left(\tau_p^{\nu_i}\right).
\end{equation*}
Continuation maps in \textit{standard} Floer cohomology are isomorphisms, so by Lemma \ref{lem c is natural}
\begin{equation}\label{eq iota of c is c}
    \iota\left([c]\right)= [c_{\nu_0}] \in HF^*\left(\tau_p^{\nu_0}\right).
\end{equation}
Combining Eq. \ref{eq c=0 in inv limit} with Eq. \ref{eq iota of c is c} gives
\begin{equation*}
    0=\iota(0)=\iota\left([c]\right)= [c_{\nu_0}] \in HF^*\left(\tau_p^{\nu_0}\right)
\end{equation*}
which concludes the proof.
\end{proof}
 
\begin{rem} In the case of \textit{strongly} monotone manifolds, i.e.  symplectic manifold whose symplectic class is a multiple of the first Chern class,  the proof can be considerably simplified and the conclusion can be strengthened.  First of all, the strong monotonicity is carried over in the mapping torus of the Dehn twist thanks to a variation of Lemma \ref{lem symplectic form rational is rational in mapping torus}.  An analogous reasoning as in \cite[Lemma 9]{Sei02} implies a uniform bound on the energy of each index $k$ half-strips which is independent of the chosen almost complex structure.  This implies that for a big enough neck-stretch parameter the moduli space \[ \bigsqcup_{ \widetilde{x_+} \in c_{\nu_i}} \mathcal{M}^{\nu}_k(V,\widetilde{x_+}) \] has to be empty,  hence in general it must be null-cobordant.
\end{rem}
\newpage
\appendix

 \section{Modified Floer cohomology coincides with the standard one} \label{app:foobar}
In this appendix we give a brief description of the isomorphism between modified fixed point Floer cohomology and standard fixed point Floer cohomology. The argument is purely algebraic and was kindly explained to us by K. Ono in the Hamiltonian case.  We do not claim any originality in this reasoning and all the errors in translating it to our specific setting are ours. \\
We will suppress the index $i$ that keeps track of the almost complex structure $J^{\nu_i}$,  since it is irrelevant for our purposes here.  We start by briefly recalling the following definition and properties of the filtration.
\begin{defi} Let $F^k{C}^{\bullet}$ be the subcomplex of ${C}^{\bullet}:={C}^{\bullet}\left(\tau_p; \Lambda_{\omega}\right)$ defined as follows \[  F^k{C}^{\bullet} := \left\lbrace \sum \xi_{\tilde{x}}\cdot \tilde{x} \in {C}^{\bullet} \mid \xi_{\tilde{x}}=0 \text{ if } \mathcal{A}_{h}(\tilde{x})< r_k\right\rbrace.\]
\end{defi}
\noindent For $k\leq j$, we have $F^j{C}^{\bullet}\subseteq F^k{C}^{\bullet}$. 
\begin{lemma} The filtration defined above is complete and exhaustive.  In other words we have the following isomorphisms of $\Lambda_{\omega}$-vector spaces \begin{align*}
F^k{C}^{\bullet} &\cong \varprojlim_{j\to +\infty}F^k{C}^{\bullet}/F^j{C}^{\bullet} \\
{C}^{\bullet}&\cong \varinjlim_{k\to -\infty} F^k{C}^{\bullet}.
\end{align*}
\end{lemma}
Before stating and proving the main result of this appendix,  let us recall the following main technical lemma from \cite{FOOO10}.
\begin{lemma}\label{lem appendix technical} There exists $\delta >0$ such that \[ F^{\lambda}C^p \cap \partial \left(C^{p-1} \right) \subset \partial \left(F^{\lambda-\delta}C^{p-1} \right).\]
\end{lemma}
\begin{proof}
See \cite[Prop.  6.3.9]{FOOO10}.  Here we need that the Novikov field $\Lambda_{\omega}$ is a PID and that for each $p$, $C^p$ is finitely generated.
\end{proof}

\begin{prop}\label{prop appendix isomorphism modified and standard floer cohomology} There is a canonical isomorphism of graded $\Lambda_{\omega}$-vector spaces \[ HF^{\ast}\left(\tau_p; \Lambda_{\omega} \right)\cong \widehat{HF^{\ast}}\left( \tau_p; \Lambda_{\omega}\right).\]
\end{prop}
\begin{proof}
By definition \[ \widehat{HF^{\ast}}\left(\tau_p; \Lambda_{\omega}\right) := \varinjlim_{k\to -\infty} \varprojlim_{j\to +\infty}  H^{\ast}\left(F^k{C}^{\bullet}/F^j{C}^{\bullet} \right),\] hence we need to prove 
\[ HF^{\ast}\left(\tau_p; \Lambda_{\omega} \right)\cong\varinjlim_{k\to -\infty} \varprojlim_{j\to +\infty}  H^{\ast}\left(F^k{C}^{\bullet}/F^j{C}^{\bullet} \right).\] 
Exactness of the colimit functor implies that \[ HF^{\ast}\left(\tau_p; \Lambda_{\omega} \right)\cong\varinjlim_{k\to -\infty} H^{\ast}\left(F^k{C}^{\bullet}\right),\] so it suffices to show that the canonical map 
\begin{equation}\label{eq map into inverse limit floer cohomology}
H^{\ast}\left(F^k{C}^{\bullet}\right) \to \varprojlim_{j\to +\infty}H^{\ast}\left(F^k{C}^{\bullet}/F^j{C}^{\bullet}\right)
\end{equation}
induced by the chain maps $F^k{C}^{\bullet} \to F^k{C}^{\bullet}/F^j{C}^{\bullet}$ is an isomorphism.  We start by showing that the map in Eq. \ref{eq map into inverse limit floer cohomology} is surjective. Pick an increasing sequence $k=k_0 < k_1 < \cdots < k_i < \cdots $ such that $\lim_{i \to \infty} k_i = +\infty$.  For each index $i$,  let $[\xi_i] \in H^{\ast}\left(F^k{C}^{\bullet}/F^{k_i}{C}^{\bullet}\right)$ be a representative of an element in the inverse limit of Eq. \ref{eq map into inverse limit floer cohomology}. This means that 
\begin{itemize}
\item $\xi_i \in F^k{C}^{\bullet}$
\item $\partial \xi_i \in F^{k_i}{C}^{\bullet+1}$
\item the natural maps forming the inverse limits map $[\xi_{i+1}] \mapsto [\xi_i].$ 
\end{itemize}
By induction on $i$, we will replace $\xi_i$ with $\widetilde{\xi}_i \in F^k{C}^{\bullet}$ satisfying the following properties
\begin{itemize}
\item $[\xi_i]=[\widetilde{\xi}_i] \in H^{\ast}\left(F^k{C}^{\bullet}/F^{k_i}{C}^{\bullet}\right)$
\item $\partial \widetilde{\xi_i} = 0 \pmod{F^{k_i}{C}^{\bullet+1}}$
\item $\widetilde{\xi}_i=\widetilde{\xi}_{i-1} \pmod{F^{k_{i-1}}{C}^{\bullet}}.$
\end{itemize}
Set $\widetilde{\xi}_1:=\xi_1$. We proceed by induction and assume we constructed $\widetilde{\xi}_i$ for $i\leq l$.  Since \[[\xi_{l+1} \pmod{F^{k_l}C^{\bullet}}]=[\xi_l]=[\widetilde{\xi}_l],\] there exists $\zeta_l \in F^kC^{\bullet}$ such that \[ \xi_{l+1}-\widetilde{\xi}_l = \partial \zeta_l \pmod{F^{k_l}C^{\bullet}}.\] 
Set $\widetilde{\xi}_{l+1}:=\xi_{l+1}-\partial \zeta_l$. Notice 
\begin{align*}
\widetilde{\xi}_{l+1}-\widetilde{\xi}_{l} &=\xi_{l+1}-\partial \zeta_l-\widetilde{\xi}_{l} \pmod{F^{k_l}C^{\bullet}}\\ &=\partial \zeta_l-\partial \zeta_l \pmod{F^{k_l}C^{\bullet}}.
\end{align*}
Therefore $\widetilde{\xi}_{l+1}-\widetilde{\xi}_{l} \in F^{k_l}C^{\bullet}$.  It is easy to see that $\widetilde{\xi}_{l+1}$ enjoys the other properties we stated above.  By definition of the topology induced by the filtration,  the sequence $\{\widetilde{\xi}_i\}$ converges to an element $\widetilde{\xi}$ in $F^kC^{\bullet}$ such that $\partial \widetilde{\xi}=0$. Clearly $[\widetilde{\xi}] \in H^{\ast}\left(F^kC^{\bullet}\right)$ represents $[\xi_i] \in H^{\ast}\left(F^kC^{\bullet}/F^{k_i}C^{\bullet}\right)$.  Hence the map in \ref{eq map into inverse limit floer cohomology} is surjective.\\ We now prove injectivity. Suppose $\xi \in F^kC^{\bullet}$ is a cycle and for every $j \geq k$ there exist $\theta_j \in F^kC^{\bullet-1}$ such that 
\[ \xi = \partial \theta_j \pmod{F^jC^{\bullet}}.\] 
Pick an increasing sequence $k=k_0 < k_1 < \cdots < k_i < \cdots $ as before but with the additional property that for each index $i$, $k_{i+1}-k_i > \delta$, where $\delta>0$ is the constant coming from Lemma \ref{lem appendix technical}.  Set $\zeta_1 :=\theta_{k_1}$,  then $\xi - \partial \zeta_1 \in F^{k_1}C^{\bullet}$.  Observe that  \[[\xi]=[\xi - \partial \zeta_1] \in H^{\ast}\left(F^kC^{\bullet}/F^{k_2}C^{\bullet}\right) \]
implies 
 \[ \xi - \partial \zeta_1\in \partial \left( F^{k}C^{\bullet-1}\right)+ F^{k_2}C^{\bullet}.\]
By Lemma \ref{lem appendix technical},   
\[ \xi - \partial \zeta_1\in \partial \left( F^{k_1-\delta}C^{\bullet-1}\right)+ F^{k_2}C^{\bullet}\subset \partial \left( F^{k}C^{\bullet-1}\right)+ F^{k_2}C^{\bullet}, \] therefore there exists $\eta_1 \in F^{k}C^{\bullet-1}$ such that 
\[ \xi - \partial \zeta_1 = \partial \eta_1 \pmod{F^{k_2}C^{\bullet}}.\] 
Set $\zeta_2 := \zeta_1 + \eta_1$,  then 
\begin{itemize}
\item $\xi= \partial \zeta_2 \pmod{F^{k_2}C^{\bullet}}$
\item $ \zeta_2-\zeta_1 \in F^{k_0}C^{\bullet-1}$.
\end{itemize}
We proceed by induction now.  Assume there are $\zeta_1, \dots , \zeta_l$ such that 
\begin{itemize}
\item $\xi = \partial \zeta_i \pmod{F^{k_i}C^{\bullet}}$
\item $\zeta_i-\zeta_{i-1} \in F^{k_{i-2}}C^{\bullet}$ for $i=2, \dots , l$. 
\end{itemize}
We want to construct $\zeta_{l+1}$. As before 
\[ \xi - \partial \zeta_l \in \partial \left( F^{k}C^{\bullet-1}\right)+ F^{k_{l+1}}C^{\bullet}\] which, together with Lemma \ref{lem appendix technical}, implies there exists $\eta_l \in F^{k_{l-1}}C^{\bullet-1}$ such that \[ \xi - \partial \zeta_l = \partial \eta_l \pmod{F^{k_{l+1}}C^{\bullet}}.\] 
Set $\zeta_{l+1} := \zeta_l + \eta_l$.  We easily see that it satisfies the required properties, hence by induction we obtain a sequence $\{\zeta_l\}$ which is convergent to an element 
\[\zeta \in \varprojlim_{l\to +\infty}F^k{C}^{\bullet-1}/F^{k_l}{C}^{\bullet-1} \cong F^k{C}^{\bullet-1}. \] Then for any index $l$,  \[ \xi =\partial \zeta \pmod{F^{k_l}{C}^{\bullet}},\] which implies $\xi = \partial \zeta$ in $F^k{C}^{\bullet}$ concluding the proof. 
\end{proof}

\newpage
\bibliographystyle{alpha}
\bibliography{bibliography}
\end{document}